\documentclass[11pt]{amsart}
\usepackage{amsmath}
\usepackage{amssymb}
\usepackage{amsfonts}
\usepackage{amsthm}
\usepackage{amsbsy}
\usepackage{amsgen}
\usepackage{amscd}
\usepackage{amsopn}
\usepackage{amstext}
\usepackage{amsxtra}

\newtheorem{theorem}{Theorem}[section]
\newtheorem{proposition}[theorem]{Proposition}
\newtheorem{lemma}[theorem]{Lemma}
\newtheorem{corollary}[theorem]{Corollary}

\theoremstyle{definition}

\numberwithin{equation}{section}

\newcommand {\Z} {\mathbb{Z}}
\newcommand {\R} {\mathbb{R}}
\newcommand {\C} {\mathbb{C}}

\newcommand {\HH} {\mathbb{H}}

\newcommand{\disc}{\operatorname{disc}}
\newcommand{\sign}{\operatorname{sign}}

\newcommand{\intinf}{\int_{-\infty}^\infty}

\newcommand{\Vol}{\operatorname{vol}}

\newcommand{\tr}{\operatorname{tr}}

\newcommand{\SL}{\operatorname{SL}}
\newcommand{\PSL}{\operatorname{PSL}}
\newcommand{\hilb}{\mathcal H}
\newcommand{\trs}{\mathbf w}
\newcommand{\ors}{\mathbf r}
\newcommand{\UU}{\mathcal U}
\newcommand{\maasK}{{\mathsf K}}
\newcommand{\maasL}{{\mathsf \Lambda}}
\newcommand{\EB}{\operatorname{Beta}} 

\begin{document}
\title[Variance of arithmetic measures on the modular surface]
{The variance of arithmetic measures associated to closed geodesics on the modular surfaces}
\author{Wenzhi Luo, Ze\'ev Rudnick and Peter Sarnak}

\address{Department of Mathematics,  The Ohio State University, 
100 Math Tower,  231 West 18 th Avenue,  Columbus, OH 43210-1174 USA}

\address{Raymond and Beverly Sackler School of Mathematical Sciences,
Tel Aviv University, Tel Aviv 69978, Israel and  
School of Mathematics, Institute for Advanced Study, 
Einstein Drive, Princeton, NJ 08540 USA}

\address{Department of Mathematics, Princeton University, Fine Hall, 
Washington Road, Princeton, NJ 08544 and 
School of Mathematics, Institute for Advanced Study, 
Einstein Drive, Princeton, NJ 08540 USA}

\date{April 15, 2009}
\maketitle
\tableofcontents

\section{Introduction}
\subsection{Equidistribution theorems for closed geodesics}

Let $X$ be a compact surface with a metric of constant negative
curvature $\kappa=-1$, $SX$ be the unit tangent bundle of $X$, and
$\Phi^t:SX\to SX$ the geodesic flow. We think of $SX$ as the set of
initial conditions $(z,\zeta)$ with $z\in X$ the position and $\zeta$
the direction vector.

The geodesic flow is {\em ergodic} with respect to Liouville measure
$dx$, the smooth invariant measure for the flow: {\em Generic}
geodesics become equidistributed, in the sense that
for Lebesgue-almost all initial conditions $x_0\in SX$,
$$
\lim_{T\to\infty} \frac 1T\int_0^T F(\Phi^t x_0) dt = \int_{SX} F(x) dx
$$
for integrable observables on $SX$.

As is well known ,  there are infinitely many
{\em closed} geodesics, in fact the number $\pi(T)$ of  
closed geodesics of length at most $T$ grows exponentially
with $T$:  $\pi(T) \sim e^{T}/T$ as $T\to \infty$ \cite{Selberg56}, \cite{Huber}. 
For a closed geodesic $C$, let $\ell(C)$ be its length and
$\mu_C$ be the arc-length measure along $C$, i.e. for $F\in C(X)$,
$$
\int_C F d\mu_C :=\int_0^{\ell(C)} F(\Phi^t x) dt,\quad
x\in C \;.
$$
This is a measure on $SX$, invariant under the geodesic flow and of
total mass $\ell(C)$.
Closed geodesics become, on average, uniformly distributed
with respect to $dx$:\footnote{In {\em variable} negative
curvature, one needs the Bowen-Margulis measure here. To get an
equidistribution statement involving Liouville measure, one needs to
weigh each geodesic by its ``monodromy''. }
For any observable $F\in C(SX)$ we have
$$
\lim_{T\to\infty} \frac 1{\pi(T)} \sum_{\ell(C)\leq T}
\frac 1{\ell(C)} \int_C F d\mu_C  = \int_{SX}F(x) dx \;.
$$


Lalley \cite{Lalley} determined the fluctuations of the numbers
$\mu_C(F)/\sqrt{\ell(C)}$ for $F$ as above of zero mean,
as $C$ varies over closed geodesics ordered by length.
He showed that they are Gaussian with mean zero and
variance $V(F,F)$ where $V$ is the hermitian bi-linear form on functions
of zero mean given by
\begin{equation}
V(F_1,F_2) =  \int_{-\infty}^\infty \left(\int_{SX} F_1(x) F_2(\Phi^t x)
 dx \right)dt \;.
\end{equation}
The negative curvature guarantees that the correlations in the inner integral decay 
exponentially as $t\to\pm \infty$, so that $V$ is convergent \cite{Ratner}.

 The bilinear form $V$ is positive semi-definite, and is degenerate,
in fact if $F_0$ is smooth then $V(F_0,F) = 0$ for all $F$ if and only if $F_0$ is a derivative in the
flow direction: $F_0 = \frac{d}{dt}|_{t=0}f\circ \Phi^t$ for some
other observable $f\in C^\infty(SX)$.
 

An important involution of $SX$ is {\em time reversal symmetry}
$$\trs: (z,\zeta)\mapsto (z,-\zeta)$$
which reverses the direction vector of the initial condition, and
satisfies $\trs \circ \Phi^t = \Phi^{-t}\circ \trs$. It induces an
involution on the
set of geodesics, taking a geodesic $C=\{\Phi^t x_0: t\in \R\}$ to
its time reversal $\bar C= \trs C = \{\Phi^s \trs x_0: s\in \R\}$.
 
Time reversal symmetry can also be incorporated in Lalley's theorem:
To do so, note that for a closed geodesic $C$,  its time-reversed
partner $\bar C$  is also closed and both have the same length:
$\ell(C)=\ell(\bar C)$. Grouping these together yields the
measure $\mu_C^{even} :=\mu_C+\mu_{\bar C}$ of mass $2\ell(C)$. By
Lalley's theorem, the fluctuations of $\mu_C^{even}/\sqrt{2\ell(C)}$
are again Gaussian with mean zero but with variance given by the
hermitian form
\begin{equation}
V^{even}(F_1,F_2) = V\left(F_1^{even}, F_2^{even}\right)
\end{equation}
where $F^{even}=(F+F\circ \trs)/2$ is the even part of $F$ under
$\trs$.
Note that $\mu_C^{even}$ is invariant
and $V^{even}$ is bi-invariant under 
the geodesic flow as well as under time-reversal symmetry $\trs$. 
Both hermitian forms $V$ and $V^{even}$ on 
$$ L^2_0(SX):= \{f\in L^2(SX): \int_{SX} f(x)dx=0 \}$$
can be diagonalized and computed explicitly by decomposing the regular representation 
of $\PSL_2(\R)$ on this space, see \S~\ref{ladder}.

\subsection{The modular surface}
In this paper we investigate fluctuations of measures on the
{\em modular surface} associated with grouping together geodesics of
equal {\em discriminant}. As is well known,
any of our compact surfaces $X$ may be uniformized as a quotient of
the upper half-plane $\HH$, equipped with the hyperbolic metric, by a
Fuchsian group $\Gamma$.
Furthermore, the group $G=\PSL_2(\R)$ of orientation preserving
isometries of $\HH$ acts transitively on the unit tangent bundle
$SX$, giving an identification $SX\simeq \Gamma\backslash G$, reviewed
in \S~\ref{sec:periodconstrut}.
The {\em modular surface} is obtained by taking $\Gamma=\PSL_2(\Z)$; the
resulting surface is non-compact (but of finite volume) and has
elliptic fixed points, but these issues will not be important for us.

Closed geodesics correspond to (hyperbolic) conjugacy classes in
$\Gamma$, with the length of a closed geodesic $C$ given in terms of the
trace $t$ of the corresponding conjugacy class  by
$\ell(C) = 2\log(t+\sqrt{t^2-4})/2$.
In the case of the modular surface, the hyperbolic conjugacy classes
correspond to (strict) equivalence classes of integer binary quadratic
forms $ax^2+bxy+cy^2$ (also denoted by $[a,b,c]$), of positive
discriminant $d:=b^2-4ac$
with the modular group acting by linear substitutions (we need to
exclude discriminants which are perfect squares).
The discriminant $\disc(C)$ of a closed geodesic $C$ is defined as the
discriminant  of the corresponding binary quadratic form.

For $d>0$,  $d\equiv 0,1\mod 4$ and $d$ not a perfect square, let
$\bar f_1,\dots \bar f_{H(d)}$ be the classes of binary quadratic
forms of discriminant $d$. We do not assume that $f_j = [a_j,b_j,c_j]$
is primitive and so $H(d)$ is the Hurwitz class number \cite{Landau}. Let
\begin{equation}
\epsilon_d = \frac{t_d+\sqrt{d}u_d}2, \quad t_d>0, u_d>0
\end{equation}
be the fundamental solution of the Pellian equation $t^2-du^2=4$. Then
as in \cite{Sa2, Sa3} associate to each $\bar f_j$ the
$\Gamma$-conjugacy class (it is well-defined) of the matrix
\begin{equation}
\begin{pmatrix}
\frac{t_d-b_j u_d}2& a_j u_d \\ -c_j u_d& \frac{t_d+b_j u_d}2
\end{pmatrix} \;.
\end{equation}
This gives  $H(d)$ closed geodesics for each discriminant $d$,
all of length $2\log \epsilon_d$.
Let $\mu_d$ be the corresponding measure on $SX$:
\begin{equation}
 \mu_d = \sum_{\disc(C)=d} \mu_C \;.
\end{equation}


These measures are the arithmetic measures in the title of the paper. 
They have been studied extensively and the primary result about them is that they 
become equidistributed as $d\to \infty$. 
That is, if $F$ is bounded and continuous on $SX$ and has mean zero, then 
$$ \frac{\mu_d(F)}{H(d)2\log \epsilon_d} \to 0, \quad \mbox{as } d\to \infty  \;.
$$

Linnik \cite{Linnik} developed an ergodic theoretic approach to this 
equidistribution problem and recently \cite{ELMV} have shown that this method 
leads to a proof of  this specific result. 
The first proof of  equidistribution is due to Iwaniec \cite{Iwaniec-eq} and Duke \cite{Duke}. 
Iwaniec established the requisite estimate for Fourier coefficients 
of holomorphic half-integral weight forms (of weight $>5/2$) and Duke obtained the estimates 
for weight $3/2$  and weight zero Maass forms. 
In view of our reductions in Sections \S~\ref{ladder} and \S~\ref{sec theta},  
these together imply\footnote{Specifically, by \eqref{eq 3.24'}, \eqref{mu as fourier coeff}  and \eqref{KS coeff formula},   equidistribution on $SX$ is reduced to estimation of Fourier coefficients of classical holomorphic forms of half integer weight and Maass forms of weight $1/2$.} 
 the full equidistribution on $SX$. 

The measures $\mu_d$ enjoy some symmetries (see \cite{Sa3}). Firstly they are invariant under 
time reversal symmetry: $\trs \mu_d = \mu_d$. 
Secondly, let $\ors$ be the involution of $\Gamma\backslash G$ 
given by $g\mapsto \delta^{-1} g \delta$, 
where $$\delta = \begin{pmatrix}1&0\\0&-1\end{pmatrix}$$ 
(it is well defined since $\delta^{-1}\Gamma \delta = \Gamma$).
In terms of the coordinates $(z,\zeta)$ on $SX$, $\ors$ is 
the orientation-reversing symmetry 
$$\ors:(z,\zeta)\mapsto (-\overline{z},-\overline{\zeta})\;.$$
The measure $\mu_d$ is also invariant under $\ors$. 
The involutions $\trs$, $\ors$ commute and their product $\ors \trs$ 
is also an involution. Thus $\mu_d$ is invariant under the Klein four-group 
$H=\{I, \ors,\trs,\ors\trs \}$. 
These involutions induce linear actions on $L^2(\Gamma\backslash G)$ by 
$f(x)\mapsto f(h(x))$ with $h\in H$ and $x\in \Gamma\backslash G$ and 
we denote these transformations by the same symbols. 
The fluctuations of the measures $\mu_d$ inherit these symmetries 
and since we are particularly interested in comparing their variance with the classical variance $V$ 
we define the symmetrized classical variance $V^{sym}$ on functions of mean zero on 
$\Gamma\backslash G$ by 
\begin{equation}
 V^{sym}(F_1,F_2):= V(F_1^{sym},F_2^{sym})
\end{equation}
where $$F^{sym}:= \frac 14 \sum_{h\in H} hF \;.
$$

\subsection{Results}

We can now state our main results about the fluctuations of $\mu_d$. We normalize these measures as 
$$ \tilde \mu_d:=\frac{\mu_d}{d^{1/4}}\; .$$
This is essentially equivalent to normalizing by the square root of the total mass, 
$\sqrt{H(d)2\log \epsilon_d}$, see Remark \ref{rem:normalize}. 
The space of natural observables for which one might compute these quantities 
is $L^2_0(\Gamma\backslash G)$, or at least a dense subspace thereof. This space decomposes as an 
orthogonal direct sum of the cuspidal subspace 
$$
L^2_{cusp}(\Gamma\backslash G):=
\{ f\in L^2(\Gamma\backslash G): \int_{N\cap \Gamma\backslash N} f(nx)dn, 
\mbox{ for a.e. }x\in \Gamma\backslash G \}
$$
where $N=\{\begin{pmatrix} 1&u\\0&1  \end{pmatrix}:u\in \R \}$, 
and the unitary Eisenstein series \cite{Godement}.  
The former is the major and difficult part of the space 
$L^2_0(\Gamma\backslash G)$ so we will concentrate exclusively on it. 
One can easily   extend our analysis of the variance to the unitary Eisenstein series.



\begin{theorem}\label{thm 1}
Fix smooth, $K$-finite $F_1,F_2\in L^2_{cusp}(\Gamma\backslash
G)$.
Then
\begin{equation}
\lim_{Y\to \infty}
\frac 1{\#\{d:d\leq Y\}} \sum_{d\leq Y} \frac{\mu_d(F_1)}{d^{1/4}} = 0
\end{equation}
and there is a limiting variance
\begin{equation}\label{eq ii of thm 1}
B(F_1,F_2)=\lim_{Y\to \infty}
\frac 1{\#\{d:d\leq Y\}} \sum_{d\leq Y} \frac{\mu_d(F_1)}{d^{1/4}}
\frac{\overline{\mu_d(F_2)}}{d^{1/4}}
\end{equation}
\end{theorem}

We call this variance $B$ the ``arithmetic variance''.
The structure of the bilinear form $B$
is revealed by choosing a special basis of observables, compatible
with the symmetries of the problem.
Recall that the unit tangent bundle is a homogeneous space for
$G=\PSL_2(\R)$, and thus it is natural to decompose the space
$L^2(\Gamma\backslash G)$ into the irreducible components under the
$G$-action.  In addition, there is an
algebra of Hecke operators acting on this space, commuting with the
$G$-action, hence also acting on each isotypic $G$-component.
We take observables lying in irreducible spaces for the joint action of
$G$ and the Hecke operators - the automorphic subrepresentations of
$L^2_{cusp}(\Gamma\backslash G)$.
Denote  the decomposition of the regular
representation on $L^2_{cusp}(\Gamma\backslash G)$ into $G$- and
Hecke-irreducible subspaces by
\begin{equation}
L^2_{cusp}(\Gamma\backslash G) = \bigoplus_{j=1}^\infty W_{\pi_j}
\end{equation}
so $\pi_j$ is a cuspidal automorphic representation. 


In order to describe the arithmetic variance explicitly we need a more 
detailed description of the $W_{\pi_j}$'s. To each $\pi_j$ is associated 
an even integer $k$, its weight (see \S~\ref{ladder}) which we indicate by $\pi_j^k$. 
For $k=0$ there are infinitely many $\pi_j^0$'s corresponding to Hecke-Maass 
cusp forms on $X$, while for $k>0$ there are $d_k$ 
such $\pi_j^k$ (where $d_k$ is either $[\frac k{12}]$ or $[\frac k{12}]+1$, 
depending if $k/2=1\mod 6$ or not), corresponding to holomorphic 
Hecke cusp forms of weight $k$. For $k<0$ let 
$$W_{\pi_j^k} = \overline{W_{\pi_j^{-k}}}  = \{\overline{f}: f\in W_{\pi_j^{-k}} \}$$
for $j=1,2,\dots ,d_{-k}$ and these correspond to the anti-holomorphic Hecke cusp forms. 
With these we have the orthogonal decompositions 
\begin{eqnarray}
 L^2_{cusp}(\Gamma\backslash G) &= &\sum_{j=1}^\infty W_{\pi_j} \oplus 
\sum_{k\geq 12} \sum_{j=1}^{d_k} \left( W_{\pi_j^k} \oplus W_{\pi_j^{-k}} \right)
\nonumber\\
&=&\sum_{j=1}^\infty U_{\pi_j^0} \oplus \sum_{k\geq 12} \sum_{j=1}^{d_k}  U_{\pi_j^k}\label{decomp}
\end{eqnarray}
where 
\begin{equation}\label{def of U} 
 U_{\pi_j^0} = W_{\pi_j^0},\quad \mbox{ and } \quad U_{\pi_j^k}=W_{\pi_j^k} \oplus W_{\pi_j^{-k}} 
\end{equation}

To each $\pi_j$ as above one associates an L-function $L(s,\pi_j)$ given by 
\begin{equation}
 L(s,\pi_j) = \sum_{n=1}^\infty \frac{\lambda_{\pi_j}(n)}{n^s}, \quad \Re(s)>1
\end{equation}
where $\lambda_{\pi_j}(n)$ is the eigenvalue of the Hecke operator $T_n$ acting on $W_{\pi_j}$. 
It is well known (Hecke-Maass) that $L(s,\pi_j)$ extends to an entire function 
and satisfies a functional equation relating its value at $s$ to $1-s$. 
In particular the arithmetical central value $L(\frac 12,\pi_j)$ is well defined (and real). 

\begin{theorem}\label{main thm}
 Both $V^{sym}$ and $B$ are diagonalized by the decomposition \eqref{decomp} 
and on each subspace $U_{\pi_j^k}$ we have that 
$$
B|_{U_{\pi_j^k}} = c(k) L(\frac 12,\pi_j^k) V^{sym}|_{U_{\pi_j^k}} 
$$
where $c(0)=6/\pi$ and $c(k) = 1/\pi$ if $k>0$. 
\end{theorem}

\subsection{Remarks}
\subsubsection{}
 
The hermitian forms $V^{sym}$ and $B$ can be computed explicitly on each 
$U_{\pi_j^k}$ (see \S~\ref{ladder}). 
Time-reversal symmetry $\trs$ forces $V^{sym}$ to vanish on $U_{\pi_j^k}$ 
for $k=2\mod 4$. Also orientation-reversal symmetry $\ors$ fixes the weight zero 
spaces $U_{\pi_j^0}$ and hence takes the generating vector (see \S~\ref{ladder})  
$\phi_j^0\in  \pi_j^0$ into $\pm \phi_j^0$. Corresponding to this sign 
we call $U_{\pi_j^0}$  even or odd. According to \S~\ref{ladder},  
$V^{sym}$ is completely determined on $U_{\pi_j^k}$ by its value on the generating vector; 
hence it follows that  $V^{sym}|_{U_{\pi_j}}\equiv 0$ for the odd $\pi_j^0$'s. 
In the above cases where $V^{sym}|_{U_{\pi_j^k}}$ vanishes, the sign $\epsilon_{\pi_j}$ 
of the functional equation of $L(s,\pi_j)$ is $-1$ and hence the central L-value 
$L(\frac 12,\pi_j)=0$ for reasons of symmetry.  
In the other cases ($k=0\mod 4$ and $\pi_j^0$ even), $\epsilon_{\pi_j}=1$ and 
$V^{sym}|_{U_{\pi_j}}\neq 0$. 
One expects that in these cases $L(\frac 12,\pi_j)\neq 0$ as well. 
However if we pass from $\Gamma=\PSL_2(\Z)$ to a congruence subgroup, 
where our analysis can be carried over with similar results, 
then there will be $\pi$'s corresponding to holomorphic forms for which $L(\frac 12,\pi)=0$ 
for number-theoretic reasons, specifically the conjecture of Birch and Swinnerton-Dyer 
\cite{Wiles-Clay}. In this case the restriction of the arithmetic variance 
to such a subspace will vanish for reasons far deeper than just symmetry.

\subsubsection{}\label{rem:normalize}
The normalization $\mu_d(F)/d^{1/4}$ is natural from the arithmetic
point of view. To be consistent with the previous normalization we
should use the square root of the total mass
$\sqrt{H(d)2\log \epsilon_d}$  of the measure. By Dirichlet's class
number formula for
$d$ fundamental, when $H(d)=h(d)$ is the ordinary class number
(and similar formulae for all $d$),
\begin{equation}
h(d)\log \epsilon_d =\sqrt{d}L(1,\chi_d) \;.
\end{equation}
The fluctuations of $L(1,\chi_d)$ are mild and well-understood \cite{Elliott}
and hence the normalizations are essentially the same. In any case one
could use methods as in \cite[Chapter 26]{IK} to remove the weights
$L(1,\chi_d)$ and deduce Theorem~\ref{thm 1} with this other normalization.
 
\subsubsection{} In \S~\ref{ladder} we show, in a more abstract
context,  that the space of linear forms on an irreducible unitary 
representation of $G$ which are invariant by both the geodesic flow
and time reversal symmetry is at most one-dimensional, and how to incorporate orientation-reversal symmetry. 
This shows that the form that the arithmetic and ``classical'' variance take is
universal. That is for any family of such invariant measures, the
variance $B'$, if it exists, is determined completely in each
irreducible representation of $G$ by $B'(v_0,v_0)$, where $v_0$ is
either a spherical vector, or a lowest (or highest) weight vector in the
representation.

\subsubsection{}
The geometric problem is to order the $\mu_d$  by the {\em length}
of any of the geodesic components of the measure.
We do not know how to do this.
What we can do is to compute the variance of the
$\mu_d$'s when ordered by the discriminant $d$. 
From the arithmetic point of view this
ordering is anyway the most natural one.
For many considerations these two orderings of $\mu_d$
yield quite different answers (see \cite{Sa3}).
However for the fluctuations we believe they are similar.

The difficulty in proving the same result of the $\mu_d$'s ordered by
$t_d$ (or $\epsilon_d$) is apparent already for $F_1=F_2=f$ a
holomorphic cusp form of weight $m\equiv 0\mod 4$. In this case
according to the formula of Kohnen and Zagier \cite{KZ}, we have for $d$
a fundamental discriminant say
\begin{equation}\label{KZ formula}
\frac{| \mu_d(f) |^2} {\sqrt{d}} = \ast L(\frac 12,f\otimes \chi_d)
\end{equation}
(with $\ast$ explicit and under control). Thus we would need to
understand the averages
\begin{equation}
\sum_{t_d\leq Y} L(\frac 12,f\otimes \chi_d)
\end{equation}
The first, but big, step in this direction would be to
understand
\begin{equation}
\sum_{t\leq Y} L(\frac 12,f\otimes \chi_{t^2-4})
\end{equation}
(see \cite{Raulf} for an execution of such an analysis on a simpler problem).
This appears to be beyond the well developed techniques for averaging
special values of L-functions in families. We leave it as an
interesting open problem.

\subsubsection{}
The recent work \cite{Rudnick-Sound} giving lower bounds for moments
of special values of L-functions in families, together with \eqref{KZ formula},
shows that the fluctuations of $\mu_d(F)/d^{1/4}$ are not Gaussian,
at least not in the sense of convergence of moments.

\subsubsection{}
The arithmetic variance $B$ in Theorem~\ref{thm 1} is the same as the
quantum variance for the fluctuations of high energy eigenstates on
the modular surface that were calculated in \cite{LS} and
\cite{Zhao}. We expect that the variance for the $\mu_d$'s when
ordered by length will be the same as $B$. This would give a
semi-classical periodic orbit explanation for the singular finding
\cite{LS} that the quantum variance is $B$ rather than $V^{even}$. It
points yet again, just as for the local spectral statistics (see the survey \cite{SaAQC}), 
to the source of the singular behaviour of the quantum fluctuations in
arithmetic surfaces being the high multiplicity of the length
spectrum.
Similar phenomena are found for the quantized cat map \cite{KR, KRR}.

\subsection{Plan of the paper}
We end  with an outline of the paper and the proof of
Theorem~\ref{thm 1}.  
In \S~\ref{sec:periodconstrut} we give recall some background connecting the
dynamics on the
modular surface with the group structure on $\SL_2(\R)$. 
In \S~\ref{ladder} we show that up to a scalar multiple, there is at most one
linear
form on the smooth vectors of an irreducible unitary representation of
$\SL_2(\R)$ which is invariant under the action of the diagonal subgroup
(corresponding to the geodesic flow) and the element $\begin{pmatrix} 0&-1\\ 1&0
\end{pmatrix}$ corresponding to time-reversal symmetry. 
We show that such a linear form is determined by its value on a ``minimal''
vector - a spherical vector in the case of a principal series representation and
a lowest/highest weight vector for holomorphic/anti-holomorphic discrete series
representations. We then bring in invariance under orientation reversal and apply 
the results to show that the bilinear forms
$V^{sym}$ and $B$ are determined by their values on Maass forms and holomorphic modular forms. 

In \S\ref{sec:half integral weight}  we give present some background on
half-integral weight forms, and
in \S~\ref{sec:Rankin Selberg}  we discuss Rankin-Selberg theory for these, giving a
mean-square result for Fourier
coefficients along positive integers by modifying work of 
Matthes \cite{Matthes1} 
for weight zero forms. 

In \S~\ref{sec theta}  we review the results of Maass \cite{Maass}, Shintani
\cite{Shintani}, Kohnen \cite{Ko1,Ko2} and
Katok-Sarnak \cite{KS}, relating periods along closed  geodesics to Fourier
coefficients of theta-lifts. 
This allows us to express $\mu_d(F)$ in terms of Fourier coefficients of
half-integral weight forms on $\Gamma_0(4)$; the precise normalizations in terms
of the inner products of the forms and their $\theta$-lifts are crucial here. 
This is where the factor $L(\frac 12, \pi)$ appears. 
These results put us in a position to use the Rankin-Selberg theory of
\S~\ref{sec:Rankin Selberg}
to determine the variance $B$, which we do in \S~\ref{sec:proof of main thm}.

\subsection{Acknowledgments} 
We would like thank Akshay Venkatesh
for insightful discussions of the material related to this paper. 
Supported by NSF FRG Grant DMS-0554373 (Sarnak and Luo) and
by the Grant No 2006254 (Sarnak and Rudnick) from the United 
States-Israel
Binational Science Foundation (BSF), Jerusalem, Israel. 

\section{Background on periods} \label{sec:periodconstrut}

\subsection{The upper half-plane and its unit tangent bundle} 
We recall the hyperbolic metric on the tangent bundle of the upper
half-plane $\HH=\{z=x+iy:y>0\}$. 
We identify  the tangent space at $z\in \HH$ with the complex
numbers: $T_z\HH \simeq\C$. The hyperbolic metric on $T_z\HH$ is then given by 
$$\langle \xi,\eta\rangle_z := \frac{\Re(\xi\bar{\eta})}{y^2} $$
and the unit tangent bundle $S\HH$ is then identified with 
$$ \{(z,\zeta)\in \HH\times \C: |\zeta|=\Im(z) \}$$

\subsubsection{Isometries} 
A unimodular matrix $g=\begin{pmatrix}a&b\\c&d\end{pmatrix}\in
\SL_2(\R)$ acts  on the upper half-plane $\HH$ via $z\mapsto
(az+b)/(cz+d)$. Set 
$$j(g,z) = cz+d$$
The differential of the map is $g'(z)=(ad-bc)/(cz+d)^2
= (cz+d)^{-2}=1/j(g,z)^2$. The induced map on the tangent bundle
$T\HH$ is then  
$$(z,\xi)\mapsto (g(z), g'(z)\xi)$$
Note that this is an action: if $g,h\in \SL_2(\R)$ then
$g(h(z,\zeta)) = (gh)(z,\zeta)$.
A computation shows that we get an isometry of $\HH$: 
$$
\langle \xi,\eta\rangle_z  = \langle g'(z)\xi,g'(z)\eta\rangle_{g(z)}  
$$

\subsubsection{Group theory} 
Define matrices 
$$n(x)=\begin{pmatrix}
1&x\\ & 1\end{pmatrix},\quad a(y) = \begin{pmatrix}
y^{1/2}& \\ & y^{-1/2}\end{pmatrix}, \quad 
\kappa(\phi)=\begin{pmatrix}
\cos(\phi/2)& \sin(\phi/2)\\-\sin(\phi/2)&\cos(\phi/2)
\end{pmatrix}
$$
The rotation $\kappa(\phi)$ preserves the base point $i=\sqrt{-1}\in \HH$. 
Note that 
$$\kappa(\phi+2\pi) = -\kappa(\phi)$$
and thus we get the same element in $\PSL_2(\R)$. 

Setting $g_{x,y,\phi} = n(x)a(y)\kappa(\phi)$  we find 
$$g_{x,y,\phi}(i,i)=(x+iy,iye^{i\phi})$$
so that using the basepoint $(i,i)\in S\HH$ of the upward pointing
unit vector at $i=\sqrt{-1}\in \HH$, 
we get a bijection
$$ \PSL_2(\R)\simeq S\HH,\qquad g\mapsto g(i,i)$$


We may then identify functions on $\PSL_2(\R)$ and on $S\HH$: If
$F(z,\zeta)$ is a function on $S\HH$ we may define $\tilde F$ on
$\SL_2(\R)$ by 
$$\tilde F(g):= F(g(i,i))$$
so that $\tilde F(g_{x,y,\phi}) = F(x+iy,iye^{i\phi})$. 

\subsubsection{Geodesics} 
The geodesic flow on $S\HH$ is defined by $\Phi^t:(z,\zeta)\mapsto
(z(t),\zeta(t))$ being the endpoint of the (unit speed) 
geodesic starting at $z$ in direction $\zeta=iye^{i\phi}$. 
It turns out that on $\PSL_2(\R)$ the geodesic flow is multiplication
on the right by $\begin{pmatrix}e^{t/2}& \\ &e^{-t/2}\end{pmatrix}$,
that is  
$$
\Phi^t(z,\zeta) = g_{x,y,\phi}\begin{pmatrix}
e^{t/2}& \\ &e^{-t/2}\end{pmatrix}(i,i)
$$
Indeed, for an initial position $(z,\zeta)\in S\HH$, we write
$(z,\zeta) = g(i,i)$ and the the geodesic 
$\vec\gamma(t)=\Phi^t(z,\zeta)$ starting at $(z,\zeta)$ will
be the translate by $g_{x,y,\phi}$ of the geodesic $\vec \gamma_0(t)$ starting at the  initial
condition $(i,i)$: $\vec \gamma(t) = g_{x,y,\phi}\vec \gamma_0(t)$.  
A computation shows that 
$$
\vec\gamma_0(t)=(e^ti,e^ti) = 
\begin{pmatrix} e^{t/2}&\\&e^{-t/2}\end{pmatrix}(i,i)
$$
and therefore 
$$
\vec\gamma(t) = g_{x,y,\phi} \begin{pmatrix} e^{t/2}&\\&e^{-t/2}\end{pmatrix}(i,i)
$$

\subsubsection{Time-reversal symmetry} 
A fundamental symmetry of phase space $S\HH$ is time reversal
$(z,\zeta)\mapsto (z,-\zeta)$. Using it, one has a symmetry of the set
of geodesics, corresponding to reversing the orientation. 
In $\PSL_2(\R)$ it is given as $g\mapsto
gw$, where $w=\begin{pmatrix} 0&-1\\1&0\end{pmatrix}$. 
Indeed, if $(z,\zeta)=g(i,i)\in S\HH$ then 
$$(z,-\zeta) = g(i,-i) = g\begin{pmatrix} 0&-1\\1&0\end{pmatrix}(i,i)$$ 

\subsubsection{Orientation reversal} 
Another fundamental symmetry is orientation reversal 
$(z,\zeta)\mapsto (-\overline{z},-\overline{\zeta})$. On $\PSL_2(\R)$ it is
given by the map 
$$
g\mapsto \delta g\delta, 
\quad \delta = \begin{pmatrix} 1&0\\0&-1 \end{pmatrix}
$$

\subsubsection{$K$-types} 
Let $k$ be an integer. 
Suppose that $F:S\HH\to \C$ satisfies 
$$F(z,e^{i\alpha}\zeta) = e^{ik\alpha}F(z,\zeta)
$$
Then the corresponding function $\tilde F$ on $\PSL_2(\R)$ satisfies 
$$\tilde F(g \kappa(\alpha)) = e^{ik\alpha} \tilde F(g)$$
that is transforms under under the right action of the maximal compact
$K=SO(2)/\{\pm I\}$ via the
character $\kappa(\alpha)\mapsto e^{ik\alpha}$. 
As an example we start with a function $f$ on $\HH$ and define
$F_f(z,\zeta)=\zeta^kf(z)$. 


\subsection{Quotients}
Let $\Gamma\subset \PSL_2(\R)$ be a Fuchsian group 
, $M=\Gamma\backslash \HH$ and $SM$ the unit tangent bundle
to $M$. The identification $S\HH \simeq \PSL_2(\R)$ descends to an
identification  
$$
SM\simeq \Gamma\backslash S\HH \simeq \Gamma\backslash \PSL_2(\R)
$$

\subsubsection{Automorphy conditions} 
Let $k\geq 0$ be an integer, and $f:\HH\to \C$ is a function on the
upper half-plane satisfying the (weak) automorphy condition
\begin{equation}\label{aut cond}
f(\gamma(z)) = (cz+d)^{2k} f(z), \qquad 
\forall \gamma =\begin{pmatrix} a&b\\c&d  \end{pmatrix}\in \Gamma
\end{equation}
We define $F_f$ on $S\HH$ by 
$$F_f(z,\zeta):=\zeta^k f(z)$$
Then 
$$ F_f(\gamma(z,\zeta)) = F_f(z,\zeta),\qquad \forall \gamma\in \Gamma$$
that is $F_f$ is $\Gamma$-invariant so descends to a function on
$SM=\Gamma\backslash S\HH$, and via the identification $F\mapsto
\tilde F$ gives a $\Gamma$-invariant function $\tilde F_f$ on
$\PSL_2(\R)$: 
$$ \tilde F_f(\gamma g) = \tilde F_f(g),  \qquad \forall \gamma\in \Gamma$$

Moreover, $F_f$ has $K$-type $k$ since from the definition we find 
$$F_f(z,e^{i\alpha} \zeta) = (e^{i\alpha}\zeta)^k f(z) = e^{ik\alpha}
F_f(z,\zeta) $$
and therefore the function $\tilde F_f$ on the group $\PSL_2(\R)$
transforms under the right action of $K=SO(2)/\{\pm I\}$ by the
character $\kappa(\alpha)\mapsto e^{ik\alpha}$.

\subsubsection{Closed geodesics on $M$} 
We consider closed geodesics on $M$, that is an initial condition
$(z_0,\zeta_0)\in S\HH$ so that there is some $T>0$ and $\gamma\in
\Gamma$ with 
$$\Phi^T(z_0,\zeta_0) = \gamma(z_0,\zeta_0)$$
Writing $(z_0,\zeta_0) = g_0(i,i)$ for a unique $g_0\in \PSL_2(\R)$ we
find that 
$$ 
\Phi^T(z_0,\zeta_0) = 
g_0 \begin{pmatrix}e^{T/2}&\\&e^{-T/2}\end{pmatrix}(i,i) 
=\gamma g_0(i,i)
$$ 
and hence that
\begin{equation}\label{eq:gamma_0 and g_0}
\gamma = \pm g_0 \begin{pmatrix}e^{T/2}&\\ &e^{-T/2}\end{pmatrix} g_0^{-1}
\end{equation}
(the equality is in $\PSL_2(\R)$, that is the matrices agree up to a
sign). 

Changing the initial condition $(z_0,\zeta_0)$ to a
$\Gamma$-equivalent one $(z_1,\zeta_1)=\delta (z_0,\gamma_0)$,
$\delta\in \Gamma$ (so that we get the same point in
$SM=\Gamma\backslash S\HH$) replaces $\gamma$ by its conjugate
$\delta\gamma\delta^{-1}$. Thus we get a well-defined conjugacy class
$\gamma_C$ corresponding to the geodesic $C$. The conjugacy class
is {\em hyperbolic} as its trace satisfies $|\tr \gamma_C|  = 2\cosh(T/2)>2$.

\subsection{A correspondence with binary quadratic forms} 

An binary quadratic form $f(x,y)=ax^2+bxy+cy^2$ (also denoted by $[a,b,c]$) 
is called {\em integral} if $a,b,c$ are integers, and is {\em primitive} if
$\gcd(a,b,c)=1$. The discriminant of $f$ is $b^2-4ac$. 
The modular group $\SL_2(\Z)$ acts on the set of integral binary
quadratic forms by substitutions, and preserves the discriminant. 

There is a bijection between $\SL_2(\Z)$-equivalence classes of
primitive binary quadratic forms of positive (non-square)
discriminant and primitive hyperbolic conjugacy classes in
$\PSL_2(\Z)$ defined as follows: 
Given a primitive hyperbolic element 
$$ \gamma=\begin{pmatrix}a&b\\c&d\end{pmatrix}$$
the corresponding form is 
\begin{equation}\label{def of B}
 B(\gamma)=\frac{\sign(a+d)}{\gcd(b,d-a,-c)}[b,d-a,-c]
\end{equation}
which is primitive by definition, and has discriminant
$$
\disc(B(\gamma))= \frac{( \tr\gamma )^2-4}{\gcd(b,d-a,-c)^2}
$$
Moreover 
$$ B(-\gamma)=B(\gamma), \qquad B(\gamma^{-1}) = - B(\gamma)$$

Given a primitive integral binary quadratic form $f=[a,b,c]$ of positive
non-square discriminant $d=b^2-4ac$, let $(t_0,u_0)$ be the
fundamental solution of the Pell equation $t^2-du^2=4$ with
$t_0>0$, $u_0>0$ (which exists since we  assume $d>0$ is not a perfect
square). Define the matrix 
$$ 
\gamma(f):=
\begin{pmatrix}\frac{t_0-bu_0}2&au_0 \\-cu_0& \frac{t_0+bu_0}2\end{pmatrix}
$$
which is hyperbolic of trace $t_0=\sqrt{du^2+4}>2$, and is primitive. 
Then  $B(\gamma(f))=f$ and gives a bijection between primitive
hyperbolic conjugacy classes in $\PSL_2(\Z)$ and equivalence classes
of primitive binary quadratic forms of non-square positive discriminant. 

\subsection{Periods} 
Consider a (primitive, oriented) closed geodesic  on $M$; 
it is determined by a primitive
hyperbolic conjugacy class  $\gamma\in \Gamma$, 
Let $C$ be the lift of the closed geodesic to to the unit tangent bundle $SM$.
For any function $F$ on $SM$, we define the period of $F$ along
$C$ by choosing a point on the lifted geodesic
$(z_0,\zeta_0)$ (that is an initial condition) and setting
$$ \int_{C} F:=\int_0^T F\circ \Phi^t(z_0,\zeta_0) dt$$
where $T>0$ is the length of the geodesic, that is the first time that
$\Phi^T(z_0,\zeta_0)=\gamma(z_0,\zeta_0)$. 

\subsubsection{An alternative expression for the period} 
For a hyperbolic matrix  $\gamma=\begin{pmatrix}a&b\\c&d\end{pmatrix}$, 
define a binary quadratic form (not necessarily primitive) 
$$Q_{\gamma}(z) = cz^2+(d-a)z-b = j(\gamma,z)\left(z-\gamma(z) \right)$$
Note that $Q_{-\gamma} = -Q_{\gamma}$. 

The two zeros $w_\pm$ of $Q_{\gamma}$ are the the fixed points of
$\gamma$, which are the intersection with real axis of
the semi-circle in the upper half-plane which determines the
closed geodesic. 
By  \eqref{eq:gamma_0 and g_0}, 
the fixed points $w_\pm$ of $\gamma$ on the boundary are 
$g_0(0)$ and $g_0(\infty)$:  
Indeed,  $\gamma(w)=w$ iff 
$$\begin{pmatrix}e^{T/2}&\\ &e^{-T/2}\end{pmatrix} g_0^{-1}(w) =
g_0^{-1}(w)
$$
that is iff $e^Tg_0^{-1}(w)=g_0^{-1}(w)$, and since $T\neq 0$ this
forces $g_0^{-1}(w) = 0,\infty$. 
Thus we find 
$$
Q_{\gamma}(z) = C(z-g_0(0))(z-g_0(\infty))
$$

Let $f:\HH\to \C$ satisfy the automorphy condition \eqref{aut cond} 
of weight $2k$ for $\Gamma$, and set $F=F_f$, that is
$F(z,\zeta)=\zeta^k f(z)$ which is a $\Gamma$-invariant function on
$S\HH$, that is a function on $SM$, 
which transforms under $SO(2)$ with K-type $k$. 
Let 
$$
r_k(f,\gamma) = \int_{z_0}^{\gamma z_0} f(z)Q_{\gamma}(z)^{k-1}dz
$$
where $z_0$ lies on the semi-circle between the fixed points of $\gamma$ 
and the contour of integration\footnote{If $f$ is holomorphic, the
integral is independent of the contour} is along the geodesic arc linking
$z_0$ and $\gamma z_0$. 
 
Let  
$$D_{\gamma}:=\tr(\gamma)^2-4=\disc(Q_{\gamma})$$
be the discriminant of the quadratic form $Q_{\gamma}$. 
Then $r_k(f,\gamma)$ is simply related to the period of $f$ on the
geodesic defined by $\gamma$ \cite[proposition 4]{Katok}:  
\begin{equation}\label{Katok formula}
r_k(f,\gamma) =
\left(-\sign(\tr(\gamma))\sqrt{D_{\gamma}}\right)^{k-1}
\int_{C} F 
\end{equation}

Therefore in terms of the corresponding binary quadratic form
$B(x,y)=B(\gamma)(x,y)$ \eqref{def of B}, we get 
\begin{equation}\label{exp for period via forms} 
\int_{C} F = \frac 1{(\disc B)^{\frac{k-1}2}} 
\int_{z_0}^{\gamma z_0} f(z) B(1,-z)^{k-1}dz =:J(B)
\end{equation}
Note that the RHS above makes sense also for non-primitive forms, and 
is dilation invariant: $J(tB)=J(B)$.

%

\section{Symmetry considerations}\label{ladder}


\subsection{Background on the representation theory of $\SL_2(\R)$} 
Let $\pi$ be an irreducible infinite dimensional unitary
representation of $\SL_2(\R)$ on a Hilbert space $\hilb$ which
factors through $G=\PSL_2(\R)$. Let $K=SO(2)$, and 
let $\hilb^{(K)}$ be the space of $K$-finite vectors in $\hilb$, that
is vectors whose translates by $K$ span a finite dimensional subspace. 
Then $\hilb^{(K)}$ is dense in $\hilb$ and consists of smooth vectors,
and the Lie algebra $\mathfrak{sl}_2$ acts on $\hilb^{(K)}$ by
$d\pi$, the differential of the action of $G$.

According to
Bargmann's classification of such $\pi$'s (we follow the
exposition in Lang \cite{Lang}), there are orthogonal
one-dimensional subspaces $\hilb_n$, with $n$ even, which are
$K$-invariant and together span $\hilb^{(K)}$. To be more precise, we
consider two cases:

i) That there is no highest or lowest $K$-type, this being the
spherical, or Maass case:
\begin{equation}
\hilb^{(K)}=\bigoplus_{n \mbox{ even}} \hilb_n
\end{equation}
with $\hilb_n$ one dimensional for $n$ even, say
$\hilb_n=\C\phi_n$, and the $\phi_n$ satisfy
\begin{equation}\label{action of ladder}
\begin{split}
d\pi(W) \phi_n &= in \phi_n \\
d\pi (E^- ) \phi_n  &= (s+1-n) \phi_{n-2} \\
d\pi (E^+ ) \phi_n &= (s+1+n) \phi_{n+2}
\end{split}
\end{equation}
where
\begin{equation}\label{lie generators} 
H= \begin{pmatrix}1&0\\0&-1\end{pmatrix}, \qquad
V=\begin{pmatrix}0&1\\1&0\end{pmatrix}, \qquad
W=\begin{pmatrix}0&1\\-1&0\end{pmatrix}
\end{equation}
are  the standard basis of the Lie algebra $\mathfrak{sl}_2(\R)$,
\begin{equation}
E^\pm = H\pm iV
\end{equation}
are in the complexified Lie algebra
$\mathfrak{sl}_2(\C)$, and $d\pi(E^\pm)$ are the weight raising/lowering
operators.
Here $s\in \C$ is a parameter which, since we assume that $\pi$ is unitary, 
lies on the imaginary axis $i\R$ or in the interval $(-1,1)$. Note that since we are assuming the
representation factors through $G=\PSL_2(\R)$, only even weights appear.

ii) $\hilb$ has a lowest or highest $K$-type. In the first case there
is an even positive integer $m_0>0$ so that
\begin{equation}
\hilb^{(K)}=\bigoplus_{\substack{m=m_0\\ m\mbox{ even}}}^\infty \hilb_m
\end{equation}
with $\hilb_m$ one-dimensional, say  $\hilb _m=\C \phi_m $
and the $\phi_m$ satisfy \eqref{action of ladder} with $s=m_0-1$.
In particular, $\phi_{m_0}$ is annihilated by the lowering operator:
\begin{equation}\label{annihilation lowest wt}
d\pi(E^-) \phi_{m_0} = 0 \;.
\end{equation}
These $\pi$'s correspond to holomorphic forms of even weight.

In the case there is a highest $K$-type, there is a negative even
integer $m_0<0$ so that
\begin{equation}
\hilb^{(K)}=\bigoplus_{\substack{m=-\infty\\m\mbox{ even}}}^{m_0} \hilb_m\;.
\end{equation}
Again $\hilb_m=\C\phi_m$ for $m\leq m_0$ even, so that
$\phi_m$ satisfy \eqref{action of ladder}  with $s = -m_0-1$
and the highest weight vector $\phi_{m_0}$ is annihilated by the
raising operator:
\begin{equation}
d\pi(E^+) \phi_{m_0} = 0 \;.
\end{equation}

In case (i) we denote by $\phi_{\pi}$ the $K$-invariant (spherical)
vector $\phi_0$. We normalize it so that $\langle \phi_0,\phi_0
\rangle =1$ and then it is unique up to multiplication by a complex
scalar of unit modulus. In case (ii) we denote by $\phi_\pi$ the
similarly normalized lowest/highest weight vector $\phi_{m_0}$.
We will call these $\phi_\pi$'s ``minimal vectors'' of the
representation. 

\subsection{Linear forms}

We consider linear forms $\eta$ on $\hilb^{(K)}$ which are
invariant by the ``geodesic flow'' and ``time reversal symmetry'',  
that is 
\begin{itemize}
\item $\eta$ is annihilated by $d\pi(H)$,
where $H= \begin{pmatrix}1&0\\0&-1\end{pmatrix} \in \mathfrak{sl}_2$
is the infinitesimal generator of the group of diagonal matrices $A$:
\begin{equation}\label{action of dpi(H)} 
\eta(d\pi(H)v)=0,\qquad \forall v\in \hilb^{(K)}
\end{equation}
In this case we say that $\eta$ is $A$-invariant\footnote{This choice
of terminology is imprecise since $\pi(A)$ need not preserve the space
of $K$-finite vectors on which $\eta$ is a-priori defined}.    
\item $\eta$ is fixed by $\pi(\begin{pmatrix}0&1\\-1&0\end{pmatrix})$
\begin{equation}\label{action of w and eta}
\eta(\pi\begin{pmatrix}0&1\\-1&0\end{pmatrix}v) = \eta( v), 
\qquad \forall v\in \hilb^{(K)}\;.
\end{equation}
In this case we say that $\eta$ is invariant under time-reversal
symmetry. 
\end{itemize}
 
\begin{proposition}
Let $\pi$ be an irreducible infinite dimensional unitary
representation of $\SL_2(\R)$ on a Hilbert space $\hilb$ which
factors through $G=\PSL_2(\R)$. Then the space of linear forms 
$\eta$ on $\hilb^{(K)}$ invariant under $A$ and $\trs$ 
is at most one-dimensional, and any such form is completely
determined by its action on a ``minimal'' vector $\phi_\pi$. 
In the case (ii) of discrete series, the space of $A$-invariant forms is 
one-dimensional, 
when $m=2\mod 4$ none of them is $\trs$-invariant, and if $m=0\mod 4$ 
then any $A$-invariant form is automatically $\trs$-invariant. 
In the spherical case the space of linear forms invariant under $A$ and $\trs$ 
is one dimensional.
\end{proposition}

This is shown by giving an explicit formula for
$\eta(\phi_n)$ in term of $\eta(\phi_\pi)$. Since the cases (i) and
(ii) have slightly different features we deal with them separately.

In case (i), $\phi_\pi=\phi_0$ is the spherical vector. We are
assuming that $\eta$ is invariant under time-reversal symmetry, that
is \eqref{action of w and eta} holds. 
Since
$$
\pi\begin{pmatrix}0&1\\-1&0\end{pmatrix} \phi_n =
-\phi_n$$
if $n\equiv 2 \mod 4$, due to \eqref{action of ladder} and
$\begin{pmatrix} 0&1\\-1&0\end{pmatrix} = \exp(\frac \pi 2 W)$,
it follows from \eqref{action of w and eta} that
\begin{equation}\label{eta on phi_n, n=2 mod 4}
\eta(\phi_n) = 0, \quad\mbox{if } n\equiv 2 \mod 4 \;.
\end{equation}
Now $2H=E^+ + E^-$ and from \eqref{action of ladder} we have
\begin{equation*}
\begin{split}
2d\pi(H)\phi_n &=\left( d\pi(E^+) + d\pi(E^-) \right) \phi_n  \\
&=(s+1-n)\phi_{n-2} + (s+1+n) \phi_{n+2} \;.
\end{split}
\end{equation*}
Hence
\begin{equation}\label{eta composed with H}
\eta(2d\pi(H)\phi_n) = (s+1-n)\eta(\phi_{n-2}) +
(s+1+n)\eta(\phi_{n+2}) \;.
\end{equation}
But the LHS of \eqref{eta composed with H} is zero since we are
assuming \eqref{action of dpi(H)}. Hence for $n$ even and in
particular $n\equiv 2\mod 4$ we have
\begin{equation*}
(n-s-1)\eta(\phi_{n-2}) = (n+s+1)\eta(\phi_{n+2}) \;.
\end{equation*}
It follows that for $m\geq 4$, $m\equiv 0\mod 4$ that
\begin{equation}\label{eta on phi_m}
\eta(\phi_m) = \eta(\phi_{-m}) = \frac{(1-s)(5-s)\cdot \dots \cdot
(m-3-s)}{(3+s)(5+s)\cdot \dots \cdot (m-1+s)} \eta(\phi_0) \;.
\end{equation}
This together with \eqref{eta on phi_n, n=2 mod 4} determines $\eta$
on $\hilb^{(K)}$ explicitly in terms of $\eta(\phi_0)$.

Conversely, \eqref{eta on phi_n, n=2 mod 4} and \eqref{eta on phi_m}
with $\eta(\phi_0)=1$ define a unique $A$- and $\trs$-invariant linear form on
$\hilb$, which we denote by $\xi_{\pi,\phi_\pi}$.
So in this case the space of such  linear forms is
one-dimensional and any such form $\eta$ satisfies
\begin{equation*}
\eta = \eta(\phi_0)\xi_{\pi,\phi_\pi} \;.
\end{equation*}

We turn to case (ii) and show that the space of $A$-invariant linear forms on
$\hilb^{(K)}$ is one-dimensional.
Consider say the lowest weight case: Take the
lowest weight vector $\phi_{m_0}$, $m_0>0$ even. 
From \eqref{action of ladder} and \eqref{annihilation lowest wt} we have
\begin{equation*}
2d\pi(H)\phi_{m_0} = \left(d\pi(E^+)+d\pi(E^-) \right) \phi_{m_0} =
2m_0\phi_{m_0+2} \;.
\end{equation*}
Hence assuming $\eta$ is $A$-invariant we get that
\begin{equation*}
\eta(\phi_{m_0+2}) = 0 \;.
\end{equation*}
Furthermore for $m>m_0$ even and by \eqref{action of ladder} we have
\begin{equation*}
(m-m_0) \eta(\phi_{m-2}) = (m+m_0) \eta(\phi_{m+2}) \;.
\end{equation*}
Hence
\begin{equation}\label{lowest wt case 2 mod 4}
\eta(\phi_m)=0,\qquad \mbox{ for } m\geq m_0,\quad m\equiv m_0+2\mod 4
\end{equation}
and
\begin{equation}\label{lowest wt case 0 mod 4}
\eta(\phi_{m_0+k}) = \frac{1\cdot 3 \cdot 5\cdot\dots \cdot(\frac
k2-1)}{(m_0+1)\cdot (m_0+3)\cdot \dots \cdot (m_0+\frac k2 -1)}
\eta(\phi_{m_0})
\end{equation}
for  $k\equiv 0\mod 4$, $k\geq 4$.

Thus the space of $A$-invariant linear forms on $\hilb^{(K)}$ is
one-dimensional. It is spanned by $\xi_{\hilb_\pi,\phi_\pi}$ where
$\xi_{\hilb_\pi,\phi_\pi}(\phi_{m_0})=1$ and is defined by
\eqref{lowest wt case 2 mod 4} and \eqref{lowest wt case 0 mod
4}. Again any $A$-invariant linear form $\eta$ on $\hilb^{(K)}$ satisfies
\begin{equation}\label{unique determination lowest}
\eta = \eta(\phi_\pi) \xi_{\hilb_\pi,\phi_\pi} \;.
\end{equation}
The case of highest weight vectors and $A$-invariant forms is the
same.

If we now impose the further condition  that $\eta$ be $\trs$-invariant
in these case (ii) representations, then invariance under
time-reversal symmetry gives as in
\eqref{eta on phi_n, n=2 mod 4} that
\begin{equation*}
\eta(\phi_m) =0,\qquad \mbox{for } m\equiv 2\mod 4\;.
\end{equation*}
This coupled with \eqref{lowest wt case 2 mod 4} means that if
$m_0\equiv 2\mod 4$ then $\eta=0$. That is of $m_0\equiv 2\mod 4$ then
there is no non-zero linear form invariant under $A$ and $\trs$.

If $m_0\equiv 0\mod 4$ then from our discussion, every $A$-invariant
linear form is automatically $\trs$-invariant and in this case such 
linear forms satisfy \eqref{unique determination lowest}.

\subsection{Orientation reversal symmetry}
We now examine the role of an additional possible symmetry, ``orientation reversal'' 
$\ors$. 
It need not act on irreducible representations of $\PSL_2(\R)$. 
What we do is given an irreducible unitary representation $\pi$ on a Hilbert space $\hilb$, we 
consider Hilbert spaces $\UU$ which in the spherical case is the original representation $\hilb$,  
and in the case of the discrete series $\hilb^m$ where there is a lowest weight vector of weight $m>0$,  
we define 
$$\UU =\hilb^{+m} \oplus \hilb^{-m}$$ 
to be the direct sum of the irreducible representations with lowest weight $m$ and that with 
highest weight $-m$. We write $\UU^{(K)}$ for the dense subspace of $K$-finite vectors in $\UU$.

An orientation-reversing symmetry of $\UU$ is a unitary map $\ors$ of $\UU$ which is an involution, that is 
\begin{equation}\label{r2I}
 \ors^2=I
\end{equation}
satisfying 
\begin{equation}\label{rW}
 \ors\pi(W) = -\pi(W)\ors 
\end{equation}
and
\begin{equation}\label{rE+}
 \ors\pi(E^+)=\pi(E^-)\ors \;.
\end{equation}
As a consequence of \eqref{rE+} and \eqref{r2I} we have
\begin{equation}\label{rE-}
 \ors\pi(E^-)=\pi(E^+)\ors \;.
\end{equation}
Moreover $\ors$ commutes with the $A$-action whose infinitesimal generator is $H=\frac 12(E^+ +E^-)$ 
by \eqref{rE+}, \eqref{rE-},  
and with time-reversal symmetry, that is with\footnote{The first $\pi$ is the constant $3.1415\dots$!}  
$ \exp(\frac \pi 2 \pi(W))$  by virtue of \eqref{rW}. 

As our basic example we consider the orientation reversal involution 
on the function space $L^2(\Gamma\backslash G)$ given by 
$$ 
\ors f(x):=f(\delta x\delta^{-1}), \quad 
\delta=\begin{pmatrix}1&0\\0&-1 \end{pmatrix} \;.
$$
The relations \eqref{rW}, \eqref{rE+} hold since for the Lie algebra elements $H$, $V$ and $W$ of \eqref{lie generators} 
we have 
$$ 
\delta W=-W\delta,\quad \delta H=H\delta, \quad \delta V=-V\delta \;.
$$

\subsection{Action of $\ors$ on weight vectors} 
We first note that due to the commutation relation \eqref{rW}, 
$\ors$ must reverse weights, that is 
$$ \ors \phi_n = c_n\phi_{-n} \;. 
$$
with $|c_n|=1$ since $\ors$ is unitary, and $c_nc_{-n}=1$ since $\ors^2=I$.  
In particular, in the spherical case when there is a vector $\phi_0$ of weight $0$, we must have 
\begin{equation} \label{parity} 
 \ors \phi_0=\epsilon \phi_0, \quad \epsilon = \pm 1 \;.
\end{equation}
We say the spherical representation $\UU$ is {\em even} if if the sign is $+$, and {\em odd} if the sign is $-$. 

In the case of the discrete series representations $\UU = \hilb^{+m}\oplus \hilb^{-m}$, $m>0$,  
we choose a lowest weight vector $\phi_m\in \hilb^{+m}$ of unit length, which we call the 
minimal (or generating) vector, that $\ors \phi_m$ is a unit vector of weight $-m$, and normalize  
a choice of highest weight vector of unit length by taking 
\begin{equation}
  \phi_{-m}:=\ors \phi_m \;.
\end{equation}
We claim that the choice of minimal vector $\phi_0$ in the spherical case 
and $\phi_m$ in the discrete series case uniquely determine $\ors$. 

Indeed, starting with  $\phi_m$, the lowest weight vector for $m>0$, 
we get from \eqref{action of ladder} for $k>0$ 
$$ 
\phi_{m+2k} = \frac 1{c(s;m,k)} \pi(E^+)^k \phi_m , \quad c(s;m,k) = \prod_{j=0}^{k-1} (s+m+1+2j)
$$
and for the highest weight vector $\phi_{-m}=\ors \phi_m$
\begin{equation*} 
 \phi_{-m-2k} = \frac 1{c(s;m,k)} \pi(E^-)^k \phi_{-m} 
\end{equation*}
Therefore using \eqref{rE+}
\begin{equation*}
\begin{split}
 \ors \phi_{m+2k} & = \frac 1{c(s;m,k)} \ors \pi(E^+)^k \phi_m \\
&=\frac 1{c(s;m,k)} \pi(E^-)^k \ors \phi_m \\
&= \frac 1{c(s;m,k)} \pi(E^-)^k   \phi_{-m} = \phi_{-m-2k}
\end{split}
\end{equation*}
and likewise
\begin{equation*}
 \ors \phi_{-m-2k} = \phi_{m+2k}
\end{equation*}
That is for the discrete series $\ors$ exactly interchanges $\phi_n$ and $\phi_{-n}$:
\begin{equation}
 \ors \phi_n = \phi_{-n}, \quad |n|\geq m, \quad n=m\mod 2
\end{equation}

In the case of the spherical representations, the same analysis shows that 
\begin{equation}
 \ors \phi_n = \epsilon \phi_{-n}, \quad n\in 2\Z
\end{equation}
where $\epsilon=\pm 1$ is determined by \eqref{parity}.

\subsection{$\ors$-invariant functionals} 
If $\eta$ is a linear functional on $\UU$, invariant under the action of $A$ 
and time-reversal symmetry $\trs$, then the functional $\eta^\ors:v\mapsto \eta(\ors v)$ 
is also invariant under $A$ and $\trs$ since $\ors$ commutes with $A$ and with $\trs$. 
We wish to determine when $\eta^\ors=\eta$. 
\begin{proposition}
The space of   linear functionals on $\UU^{(K)}$ which are invariant under $A$, $\trs$ and $\ors$ 
is at most one dimensional. 
In the spherical case there are no such functionals for odd representations, 
and the space is one-dimensional in the even case, every   functional invariant under $A$ and $\trs$ 
being automatically $\ors$-invariant.  
For the discrete series there are no such functionals for weight $m=2\mod 4$, 
 and for weight $m=0\mod 4$ the space of $A$-invariant functionals is two-dimensional, 
each is automatically invariant under $\trs$ and the subspace of 
$\ors$-invariant functionals is one-dimensional. 
\end{proposition}
\begin{proof}
 We start with the spherical case. There is a one-dimensional space of functionals 
invariant under $A$ and $\trs$, and we take the unique one satisfying 
$$ \eta(\phi_0)=1\;.
$$
Hence $\eta^\ors$, being itself invariant under $A$ and $\trs$, must be a multiple of $\eta$, 
and because $\ors^2=I$ we have 
$$ 
\eta^\ors = \pm \eta\;.
$$
We claim the sign is determined by the sign in \eqref{parity}, that is if 
$\ors\phi_0 = \epsilon \phi_0$ then 
$$\eta^\ors = \epsilon \eta \;.
$$
It suffices to check this on the spherical vector $\phi_0$, that is to show $\eta^\ors(\phi_0) = \epsilon$. 
Indeed, we have 
$$ \eta^\ors(\phi_0) = \eta(\ors\phi_0) = \eta(\epsilon(\phi_0) = \epsilon \eta(\phi_0)=\epsilon$$
as required. Thus in the odd case $\eta^\ors = -\eta$ and there are no $A$- and $\trs$-invariant 
functionals which are $\ors$-invariant, and in the even case $\eta^\ors=\eta$ and every 
$A$- and $\trs$-invariant functional is automatically $\ors$-invariant.

In the discrete series case, there are no functionals invariant under $A$ and $\trs$ 
if $m=2\mod 4$, hence we only consider the case $m=0\mod 4$. 
In that case there are unique $A$-invariant functionals $\eta_+$ on $\hilb^{+m}$ 
and $\eta_{-}$ on $\hilb^{-m}$ satisfying 
$$\eta_+(\phi_m)=1, \quad \eta_-(\phi_{-m})=1$$
and these are automatically invariant under time reversal symmetry. 
Hence the space of  $A$-invariant functionals on $\UU = \hilb^{+m}\oplus \hilb^{-m}$ is two dimensional, 
consisting of linear combinations
$$ \eta =c_+ \eta_+ \oplus c_-\eta_- $$
and these are automatically invariant under time reversal symmetry. 
They are uniquely determined by their action on the  lowest and highest weight vectors 
$\phi_m$ and $\phi_{-m} = \ors \phi_m$:
$$c_{\pm} = \eta(\phi_{\pm m})$$

Since $\eta^\ors$ is also $A$-invariant, we have  
$$ \eta^\ors = c'_+\eta_+ + c'_-\eta_-$$  
Now $\eta^\ors=\eta$ if and only if $c'_+=c_+$ and $c'_-=c_-$.  
We claim that this happens if and only if $c_+=c_-$, 
which will show that the space of $A$-invariant functionals which are $\ors$-invariant 
is exactly one-dimensional in this case. 
Indeed, we have 
$$ c'_+=\eta^\ors(\phi_m) = \eta(\ors\phi_m) = \eta(\phi_{-m}) = c_-$$
$$ c'_- = \eta^\ors(\phi_{-m}) = \eta(\ors \phi_{-m}) = \eta(\phi_m) = c_+$$
and so $c'_\pm=c_\pm$ if and only if $c_+=c_-$ as claimed. 
\end{proof}

\subsection{Bilinear forms}
We apply the above uniqueness of linear forms to bi-invariant
sesqui-linear forms on $\UU\times \UU$.
Let $T(v,v')$ be such a
form, that is linear in $v$, conjugate-linear in $v'$ and invariant under 
$A$, $\trs$ and $\ors$ in each variable separately. 
For instance, we can take 
$$ T(v,v') = \sum_{j=1}^J \eta_j(v)\overline{\eta'_j(v')}$$
where $\eta_j$, $\eta'_k$  are invariant linear forms. 
 From the prior discussion,
\begin{equation*}
T(v,v')=0
\end{equation*}
if $\pi$ is of type (ii) with $m_0\equiv 2\mod 4$. Otherwise $T$ is
completely determined by value $T(\phi_\pi,\phi_\pi)$ at the minimal vector $\phi_\pi$. 
In fact $T$ is the product of linear forms
\begin{equation*}
T(v,v') = T(\phi_\pi,\phi_\pi) \xi_{\UU,\phi_\pi}(v)
\overline{\xi_{\UU,\phi_\pi}(v')} \;.
\end{equation*}
where $\xi_{\UU,\phi_\pi}$ is the unique invariant linear form taking value $1$ at the minimal vector $\phi_\pi$. 

\subsection{Application to the classical and arithmetic variances}
We apply these remarks to the measures $\mu_d$ and to 
the classical variance $V$. 
We consider the discrete decomposition of the regular representation of 
$G=\PSL_2(\R)$ on $L^2_{cusp}(\Gamma\backslash G)$. 
For an irreducible sub-representation, form the space $\UU_\pi$ as above. 

\subsubsection{} 
The arithmetic measure  $\mu_d$ is a   linear form on $\UU_\pi$ invariant 
under $A$, $\trs$ and $\ors$. 
Hence $\mu_d(F)\equiv 0$ if $\pi$ is a
discrete series with weight $m_0\equiv 2\mod 4$ and otherwise
\begin{equation}\label{eq 3.24'}
\mu_d = \mu_d(\phi_\pi) \xi_{\UU_\pi,\phi_\pi} \;.
\end{equation}
Hence if $F_1$ and $F_2$ are in $\UU_{\pi_1}$ and $\UU_{\pi_2}$ the
sesqui-linear $\mu_d$ sums take the form
\begin{equation*}
\sum_{d\leq Y} \frac{\mu_d(F_1)}{\sqrt{d}}\frac{\mu_d(F_2)}{\sqrt{d}}
= \xi_{\UU_{\pi_1},\phi_{\pi_1}}(F_1)
\overline{\xi_{\UU_{\pi_2},\phi_{\pi_2}}(F_2)}
\sum_{d\leq Y} \frac{\mu_d(\phi_{\pi_1}) \mu_d(\phi_{\pi_2})}{d} \;.
\end{equation*}
This gives a universal reduction for computation of the variance of
$\mu_d$ to the cases $F_1=\phi_{\pi_1}$, $F_2=\phi_{\pi_2}$.

\subsubsection{}  
The classical variance $V$ is by its definition diagonalized by 
the irreducibles in the decomposition of $L^2(\Gamma\backslash G)$. 
We define projections onto the set of $\trs$-invariant functions 
$$
F^{even}:= \frac 12(F+F^\trs)
$$
and onto the set of functions invariant under both $\trs$ and $\ors$, 
$$
F^{sym} := \frac 14(F+F^\trs +F^\ors +F^{\trs \ors}) = \frac 12(F^{even}+(F^{even})^\ors)
$$
Set 
$$V^{ev}(F_1,F_2) = V(F_1^{even},F_2^{even}), \quad V^{sym}(F_1,F_2) = V(F_1^{sym}, F_2^{sym})
$$
We wish to completely determine $V^{sym}$ and $V^{even}$. 

For an irreducible $\pi$, $V^{sym}$ vanishes on $\UU_\pi$ if $\pi$ is discrete series of weight 
$m_0\equiv 2\mod 4$ and otherwise is given by
\begin{equation*}
V^{sym}(v,v') = V^{sym}(\phi_\pi,\phi_\pi)
\xi_{\UU_{\pi},\phi_{\pi}}(v) \overline{\xi_{\UU_{\pi},\phi_{\pi}}(v')} \;.
\end{equation*}
It remains to determine $V^{sym}(\phi_\pi,\phi_\pi)$. 

\begin{lemma}
 i) For  $\pi$ spherical  with parameter $s=ir$,  
\begin{equation}\label{determine V^sym spherical}
 V^{sym}(\phi_0,\phi_0) =
\frac{|\Gamma(\frac 14+ir)|^4}{2\pi |\Gamma(\frac 12+2ir)|^2} \langle \phi_0,\phi_0 \rangle
\end{equation}

ii) For $\pi$ discrete series of weight $m=0\mod 4$ ($m>0$),
\begin{equation}\label{determine V^sym discrete}
 V^{sym}(\phi_m,\phi_m) =\frac 12 2^{m}\EB(\frac{m}2,\frac{m}2)
\end{equation}
where $\EB$ is Euler's beta function. 
\end{lemma}

\begin{proof}

In the spherical case we need to compute $V^{sym}(\phi_0 ,\phi_0)$. 
For spherical representations, we saw that $A$-invariance and $\trs$-invariance automatically 
imply $\ors$-invariance, hence on such spherical $\UU$, 
$$V^{sym}|_{\UU\times \UU}=V^{even}|_{\UU\times \UU}\;.
$$
Moreover, since $\phi_0$ is spherical, $\trs \phi_0=\phi_0$ and hence $\phi_0^{even}=\phi_0$. 
Thus $V^{sym}(\phi_0,\phi_0) = V(\phi_0,\phi_0)$, which was computed in \cite{LS} to give
\begin{equation*}
 V^{sym}(\phi_0,\phi_0) = V(\phi_0,\phi_0) =
\frac{|\Gamma(\frac 14+ir)|^4}{2\pi |\Gamma(\frac 12+2ir)|^2} \langle \phi_0,\phi_0 \rangle \;.
\end{equation*}

For $\pi$ discrete series of  weight $m_0\equiv 0\mod 4$ ($m>0$) 
then $A$-invariance implies $\trs$-invariance, hence 
$$V^{even}|_{\UU\times \UU}=V|_{\UU\times \UU}$$
and so  
$$V^{sym}(\phi_m,\phi_m) = 
\frac 14\left(V(\phi_m,\phi_m) + V(\phi_m,\ors \phi_m) + V(\ors \phi_m,\phi_m) + 
V(\ors\phi_m,\ors\phi_m) \right)
$$
By its definition, $V$ respects the orthogonal decomposition into irreducibles; since $\phi_m\in \pi_m$ 
and $\ors \phi_m = \phi_{-m}\in \pi_{-m}$ lie in distinct irreducibles, we get 
$V(\phi_m,\ors \phi_m) = 0 =  V(\ors \phi_m,\phi_m)$. Moreover we have 
$$ V(\ors F_1, \ors F_2) = V(F_1,F_2)$$ 
for any $F_1$, $F_2$. To see this, note first that $\ors$ is induced by the measure preserving map 
$x\mapsto \delta x\delta^{-1}$ of $SX = \Gamma\backslash G$ and hence  
$$ 
\langle \ors F_1,\ors F_2 \rangle = \int_{ \Gamma\backslash G} \ors F_1(x) \overline{\ors F_2(x)} dx = 
\langle F_1, F_2 \rangle
$$
Moreover, $\ors$ commutes with the geodesic flow and so 
\begin{equation*}
 \begin{split}
V(\ors F_1, \ors F_2) &= \intinf \langle \pi\begin{pmatrix}e^{t/2}&0\\0&e^{-t/2}\end{pmatrix} \ors F_1, \ors F_2
\rangle dt \\ 
& = \intinf \langle \ors \pi \begin{pmatrix}e^{t/2}&0\\0&e^{-t/2}\end{pmatrix}  F_1 , \ors F_2
\rangle dt \\ 
&= \intinf \langle \pi \begin{pmatrix}e^{t/2}&0\\0&e^{-t/2}\end{pmatrix}  F_1 ,   F_2 \rangle dt \\
&=V(F_1, F_2) \;.
 \end{split}
\end{equation*}
Thus we find 
$$ V^{sym}(\phi_m,\phi_m) = \frac 12 V(\phi_{-m},\phi_{-m})\;.
$$

 Let
\begin{equation*}
f(x) = \langle \pi(x) \phi_{-m},\phi_{-m} \rangle ,\qquad x\in G \;.
\end{equation*}
Then applying the raising operator $E^+$ via the regular
representation gives an operator $\mathcal L_{E^+}$ satisfying 
\begin{equation}\label{annihilated by LEplus}
\mathcal L_{E^+}f(x) = 
\langle \pi(x) (d\pi(E^+)\phi_{-m}), \phi_{-m} \rangle =0
\end{equation}
Also by the unitarity of $\pi$,
\begin{equation}\label{transformation of f}
\begin{split}
f(k(\theta_1) x k(\theta_2)) &= \langle \pi(k(\theta_1) x k(\theta_2))
\phi_{-m},\phi_{-m} \rangle  \\
&= e^{-im(\theta_1+\theta_2)} \langle \pi(x ) \phi_{-m},\phi_{-m}
\rangle = e^{-im(\theta_1+\theta_2)} f(x)
\end{split}
\end{equation}

Using the coordinates
$k(\theta_1) \begin{pmatrix}e^{r/2}&0\\0&e^{-r/2} \end{pmatrix}
k(\theta_2)$
on $G$ and the formula for $\mathcal L^+$ in these coordinates,  
we deduce from \eqref{annihilated by LEplus}  and
\eqref{transformation of f}  that
\begin{equation*}
f( k(\theta_1) \begin{pmatrix}e^{r/2}&0\\0&e^{-r/2} \end{pmatrix}
k(\theta_2)) = e^{-im(\theta_1+\theta_2)} g(r)
\end{equation*}
where $g$ satisfies the ODE
\begin{equation}\label{ODE}
2\frac{dg}{dr}  =- \frac{\cosh r}{\sinh r} m g + \frac{ m}{\sinh r} g
\end{equation}
and since we have normalized $\langle \phi_{-m},\phi_{-m}
\rangle =1$, we have
\begin{equation*}
g(0)=1 \;.
\end{equation*}
We integrate \eqref{ODE} and find that
\begin{equation*}
g(r) = (\cosh \frac r2)^{-m}
\end{equation*}
Hence
\begin{equation*} 
V( \phi_{-m},\phi_{-m}) = \intinf (\cosh \frac r2)^{-m} dr
=2^{m}\EB(\frac{m}2,\frac{m}2)
\end{equation*}
where $\EB$ is Euler's beta function. 
Thus we find 
\begin{equation*}
 V^{sym}(\phi_m,\phi_m) =\frac 12 2^{m}\EB(\frac{m}2,\frac{m}2)
\end{equation*}
\end{proof}

\section{Half-integral weight forms}\label{sec:half integral weight} 

\subsection{Basic properties}


Let $\Gamma$ be a discrete subgroup of $\SL_2(\R)$ of finite co-volume. 
Given a character $\chi:\Gamma \to S^1$, an automorphic function
of weight $k$ and character $\chi$ for $\Gamma$ is a function
$f:\HH \to \C$ satisfying
$$f(\gamma z) = \chi(\gamma) \left(\frac{cz+d}{|cz+d|}\right)^k f(z), 
\quad \forall \gamma\in \Gamma$$
with suitable growth conditions at the cusps of $\Gamma$. 
It is cuspidal if it vanishes at the cusps. 

The Laplacian of weight $k$ is defined as
$$
\Delta_k = y^2 \left( \frac{\partial^2}{\partial x^2} +
\frac{\partial^2}{\partial y^2} \right)
-iky \frac{\partial}{\partial x}
$$
The Laplacian $\Delta_k$ 
 maps forms of weight $k$ to themselves, and maps cusp forms to themselves.
A Maass cusp form of weight $k$ is a cuspidal automorphic function of
 weight $k$ (for some character $\chi$) which is an eigenfunction of
 $\Delta_k$.


Let $W_{\kappa,\mu}$ be the standard Whittaker
function, normalized so that at infinity
\begin{equation}\label{W at infinity} 
W_{\kappa,\mu}(y) \sim y^{\kappa}e^{-y/2}, \quad y\to\infty  
\end{equation}
The asymptotic behaviour of $W_{\kappa,\mu}(y)$ near $y=0$ is 
\begin{equation}\label{W at 0 a}
W_{\kappa,\mu}(y) \sim 
\frac{\Gamma(-2\mu)}{\Gamma(\frac 12 -\mu-\kappa)}y^{\frac 12+\mu} + 
\frac{\Gamma(2\mu)}{\Gamma(\frac 12 +\mu-\kappa)} y^{\frac 12-\mu}, 
\quad y\to 0
\end{equation}
for $\mu\neq 0$, while 
\begin{equation}\label{W at 0 b}
W_{\kappa,0} \ll y^{1/2} \log y, \quad y \to 0 
\end{equation}
The  functions
$$f_k^{\pm}(z,s):= W_{\pm \frac k2,s-\frac 12}(4\pi y)e(\pm x)$$
are eigenfunctions of $\Delta_k$ with eigenvalue $\lambda=s(1-s)$.

 A Maass cusp form $F$ of weight
$k$ and eigenvalue $\lambda=s(1-s)$ has Fourier expansion
$$ F(z) = \sum_{n\neq 0} \rho(n) f_k^{\sign(n)}(|n|z,s) =\sum_{n\neq 0}
\rho(n) W_{\frac {\sign(n)k}2,s-\frac 12}(4\pi |n|y)e(nx)
$$

The Petersson inner product is defined for a pair of (cuspidal) functions of the
same weight $k$ and character $\chi$, as
$$ \langle f,g \rangle =\int_{\Gamma\backslash \HH} f(z)
\overline{g(z)} \frac{dxdy}{y^2}
$$

\subsection{Maass operators} 

For any real  $k$, define the raising operator
$$
\maasK_k = \frac k2 + y\left( i\frac{\partial}{\partial x}+
\frac{\partial}{\partial y} \right) = \frac k2 +(z-\bar  z)
\frac{\partial}{\partial z}
$$
and the lowering operator
$$
\maasL_k = \frac k2 + y\left( i\frac{\partial}{\partial x} -
\frac{\partial}{\partial y} \right) = \frac k2 +(z-\bar  z)
\frac{\partial}{\partial \bar z}
$$

The raising operator  $\maasK_k$ takes Maass forms of weight $k$ to forms
of weight $k+2$ and the lowering operator $\maasL_k$ takes Maass forms
of weight $k$ to forms of weight $k-2$.

Then 
$$
\maasK_k\Delta_k = \Delta_{k+2}\maasK_k,\qquad \maasL_k\Delta_k =
\Delta_{k-2}\maasL_k
$$

The effect of the Maass operators on Petersson inner products is given
as: If $f,g$ have weight $k$ and character $\chi$, then
$$
\langle \maasK_k f,\maasK_kg \rangle = \left(\lambda(s)-\lambda(-\frac k2)
\right) \langle f,g \rangle
$$
and
$$
\langle \maasL_k f,\maasL_kg \rangle = \left(\lambda(s)-\lambda(\frac k2)
\right) \langle f,g \rangle
$$

The action of the Maass operators on the eigenfunctions $f_k^{\pm}(z,s)$ is
\begin{equation}\label{raising f^pm}
\maasK_k f_k^+(z,s) = -f_{k+2}^+(z,s),\quad
\maasK_k f_k^-(z,s) = (s+\frac k2)(1-s+\frac k2) f_{k+2}^-(z,s)
\end{equation}
and
\begin{equation}
\label{lowering f^pm}
\maasL_k f_k^+(z,s) = -(s-\frac k2)(1-s-\frac k2) f_{k-2}^+(z,s),\quad
\maasL_k f^-(z,s) = f_{k-2}^-(z,s)
\end{equation}

\subsection{Maass operators and Fourier expansions}
We want to see the Fourier expansion of a ``raised'' Maass form in
terms of its original.

So start with a Maass form $F$ of weight $1/2$ and eigenvalue
$\lambda=1/4+r^2$ 
with Fourier expansion
$$
F(z) = \sum_{n\neq 0} \rho(n) W_{\frac {\sign(n)}4,ir}(4\pi |n|y)e(nx)
$$

Applying the Maass raising operator  $\maasK_{1/2}$, we get a form
$\maasK_{1/2}F$ of weight $5/2$ whose Fourier expansion is obtained by 
\eqref{raising f^pm} as
\begin{multline}
\maasK_{1/2}F(z) = \sum_{n=1}^\infty -\rho(n) W_{5/4,ir}(4\pi ny)e(nx) \\
+\sum_{n=1}^\infty \left((\frac 34)^2+r^2 \right)
\rho(-n)W_{-5/4,ir}(4\pi ny)e(-nx)
\end{multline}
Applying the lowering operator $\maasL_{1/2}$ we get a form
$\maasL_{1/2}F$ of weight  $-3/2$ with Fourier expansion is obtained by 
\eqref{lowering f^pm} as
\begin{multline}
\maasL_{1/2}F(z) = \sum_{n=1}^\infty -\left( (\frac 14)^2+r^2 \right)
\rho(n) W_{-3/4,ir}(4\pi ny)e(nx) \\
+\sum_{n=1}^\infty \rho(-n)W_{3/4,ir}(4\pi ny)e(-nx)
\end{multline}

\section{Rankin-Selberg theory}\label{sec:Rankin Selberg} 
\subsection{Classical Rankin-Selberg theory} 
We recall classical Rankin-Selberg theory as applied to a 
holomorphic form $F$  of weight $k+1/2$ with Fourier expansion 
$$F(z) = \sum_{d\geq 1}c_F(d) e(dz) \;.
$$ 
Let $E(z,s)$ be the standard Eisenstein series for $\Gamma_0(4)$:
$$
E(z,s) = \sum_{\gamma\in \Gamma_\infty \backslash \Gamma_0(4)} \Im(\gamma z)^s
$$
where
$$
\Gamma_\infty = \{ \pm\begin{pmatrix}1&n\\0&1\end{pmatrix}:n\in \Z\}
$$
The series is absolutely convergent for $\Re(s)>1$, with an analytic
continuation  to $\Re(s)>1/2$ except for  a simple pole at $s=1$,
where the residue is
$$
\mbox{Res}_{s=1}E(z,s) =  \frac{1}{\Vol(\Gamma_0(4)\backslash \HH)}
= \frac 1{2\pi}
$$
One starts with the integral 
$$ I(s)=
\int_{\Gamma_0(4)\backslash \HH} |F(z)|^2 E(z,s) y^{k+\frac 12} \frac{dxdy}{y^2}
$$ 
which is analytic in $\Re(s)>1/2$ except for 
a simple pole at $s=1$ with residue 
$$R(F)=\frac{\langle F,F\rangle}{2\pi} 
$$ 
By the ``unfolding trick'', we have 
$$
I(s)=(4\pi)^{-(s+k-\frac 12)} \Gamma(s+k-\frac 12) 
\sum_{n=1}^\infty|\frac{c_F(n)}{n^{\frac{k-1/2}2}}|^2 n^{-s} 
$$
and hence the Dirichlet series 
$$ D(s)=\sum_{n=1}^\infty|\frac{c_F(n)}{n^{\frac{k-1/2}2}}|^2 n^{-s} $$
has a simple pole at $s=1$ with residue 
$$ \frac{(4\pi)^{k+\frac 12} }{\Gamma(k+\frac 12)} 
\frac{\langle F,F\rangle}{2\pi}  
$$
Consequently we find 
\begin{equation}\label{standard RS} 
  \lim_{N\to \infty} \frac 1N\sum_{n\leq N} |\frac{c_F(n)}{n^{\frac{k-1/2}2}}|^2
 =  \frac{(4\pi)^{k+\frac 12} }{\Gamma(k+\frac 12)} \frac{\langle F,F\rangle}{2\pi}
\end{equation}

Similar considerations show that if we take forms $F$ of weight $k+1/2$ and $G$ of weight $\ell +1/2$ 
($k$ and $\ell$ possibly different) which are orthogonal then we have 
\begin{equation}\label{standard RS orthogonal} 
  \lim_{N\to \infty} \frac 1N\sum_{n\leq N} \frac{c_F(n)}{n^{\frac{k-1/2}2}} 
\frac{\overline{c_G(n)}}{n^{\frac{\ell-1/2}2}} 
 =  0
\end{equation}
and that if $F$ is a Maass form f weight$1/2$ for $\Gamma_0(4)$ with Fourier expansion 
$$
F(x+iy) = \sum_{n\neq 0} \rho(n) W_{\sign(n)/4, ir}(4\pi |n|y) e(nx) 
$$
and $G$ is a holomorphic form of weight $k+1/2$ then 
\begin{equation}\label{standard RS orthogonal maass holom} 
 \lim_{N\to \infty} \frac 1N \sum_{n\leq N} \sqrt{n} \rho(n) 
\frac{\overline{c_G(n)} }{ n^{\frac{k-1/2}2}}  
 =  0
\end{equation}

These arguments will also give the asymptotics of the sum of squares 
$$ \sum_{-N\leq n \leq N} |4\pi n\rho(n)|^2
$$
of Fourier coefficients with both positive and negative indices. 
However for our application we need to be able to separately sum only
coefficients indexed by positive integers, that is we require the
asymptotics of the series  
$$ \sum_{n=1}^N |4\pi n\rho(n)|^2  $$
To do so, we make use of the arguments in the paper by Matthes
\cite{Matthes1}, which we adapt for our case, 
see also \cite{Maass1964, Muller1992}. 

\subsection{One-sided Rankin-Selberg theory}

Let $F(z)$  and $F'(z)$ be Maass cusp forms of weight $1/2$ for
$\Gamma_0(4)$, and Laplace eigenvalues $\lambda=\frac 14+r^2$, $\lambda'=\frac
14+(r')^2$, with
Fourier expansion
$$F(x+iy) = \sum_{n\neq 0} \rho(n) W_{\sign(n)/4, ir}(4\pi |n|y) e(nx) $$
$$F'(x+iy) = \sum_{n\neq 0} \rho'(n) W_{\sign(n)/4, ir'}(4\pi |n|y) e(nx) $$

We define two Dirichlet series
\begin{equation*}
\begin{split}
L_+(s,F\times F') &=
\sum_{n=1}^\infty \frac{4\pi n\rho(n)\overline{\rho'(n)}}{n^{s}} \\
L_-(s,F\times F') &=
\sum_{n=1}^\infty \frac{4\pi n \rho(-n)\overline{\rho'(-n)}}{ n^{s}}  
\end{split}
\end{equation*}
\begin{proposition}\label{prop:Matthes} 
Both $L_\pm(s)$ have analytic continuation to
$\Re(s)> \frac 12$,  except for a simple pole at $s=1$ if $F$ and $F'$ are
not orthogonal.
\end{proposition}

We next compute the residue at $s=1$ when $F'=F$: 
\begin{proposition}\label{prop:Rankin-Selberg}
 The residue at $s=1$ of $L^+(s,F\times F)$ is
\begin{equation}\label{form2 for R+}
R^+:=\mbox{Res}_{s=1} L_+(s,F\times F)=
\frac{|\Gamma(\frac 14+ir)|^2} {|\Gamma(\frac
12+2ir)|^2} \frac{\langle F,F \rangle}{\pi}
\end{equation}
and the residue of $L_{-}(s,F\times F)$ is 
$$
R^{-} = \frac{|\Gamma(\frac 34+ir)|^2} {|\Gamma(\frac
12+2ir)|^2} \frac{\langle F,F \rangle}{\pi}
$$
\end{proposition}
The arguments and ideas needed to establish this have been essentially
provided in \cite{Matthes1}.

As a consequence, we deduce by a standard Tauberian argument that
\begin{corollary}\label{Tauberian cor} 
Let $F$ be as above. Then
\begin{equation}
\sum_{1\leq n\leq N} 4\pi n |\rho(n)|^2 \sim R^+ N, \quad N\to \infty
\end{equation}
while if $F$ and $F'$ are orthogonal then 
\begin{equation}
\sum_{1\leq n\leq N} 4\pi n \rho(n) \overline{\rho'(n)} = o(N)
\end{equation}
\end{corollary}
\begin{proof}
 Applying the Wiener-Ikehara Tauberian theorem (see \cite{Lang-ANT} for
example) 
to the Dirichlet series
\[\sum _{n=1}^{\infty} \frac{4\pi n |\rho(n)|^{2}}{n^{s}},\;\;\;
\sum _{n=1}^{\infty} \frac{4\pi n |\rho '(n)|^{2}}{n^{s}}, \;\;\; \mbox{and}
\;\;\;\sum _{n=1}^{\infty} \frac{4\pi n |\rho(n) + \rho '(n)|^{2}}{n^{s}} \]
respectively, we infer by proposition~\ref{prop:Matthes} that, 
as $N \rightarrow \infty$,
\[\sum _{1\leq n \leq N} 4\pi n |\rho (n)|^{2} \sim R^{+}_{F} N ,\]
\[\sum _{1\leq n \leq N} 4\pi n |\rho '(n)|^{2} \sim R^{+}_{F'} N ,\]
and
\[\sum _{1\leq n \leq N} 4\pi n |\rho(n) + \rho '(n)|^{2} \sim 
(R^{+}_{F} + R^{+}_{F'}) N .\]
Thus,
\[\sum _{1\leq n \leq N} 4\pi n \;\Re (\rho(n) \overline{\rho '(n)}) = o(N) .\]
Similarly applying the Wiener-Ikehara theorem to the Dirichlet series
\[\sum _{n=1}^{\infty} \frac{4\pi n |\rho(n) + i\rho '(n)|^{2}}{n^{s}} , \]
we obtain
\[\sum _{1\leq n \leq N} 4\pi n |\rho(n) + i\rho '(n)|^{2} 
\sim (R^{+}_{F} + R^{+}_{F'}) N ,\]
and consequently,
\[\sum _{1\leq \leq N} 4\pi n \; \Im (\rho(n) \overline{\rho '(n)}) = o(N) .\]
This proves
\[\sum _{1\leq n\leq N} 4\pi n \rho(n) \overline{\rho '(n)} = o(N) .\]
\end{proof}

\subsection{A Mellin transform} 

Fix $\mu=ir$ and $\nu=ir'$ and define 
$$
\mathcal M_k(s):=\int_0^\infty W_{k,\mu}(y)W_{k,\nu}(y) y^{s-2}dy
$$
In view of the asymptotics \eqref{W at infinity}, 
\eqref{W at 0 a}, \eqref{W at 0 b},  
the integral is absolutely convergent for $\Re(s)>0$ and hence
$\mathcal M_k(s)$ is analytic in that region.
The asymptotic behaviour of $\mathcal M_k(s)$ is given by 
\begin{lemma}\label{asymp of Mk}
Assume that $|k|<1/2$. Then as $|s|\to\infty$, 
\begin{equation*}
\mathcal M_k(s) 
= \frac{\Gamma(s-\frac 12 +k)^2}{\Gamma(s)} 
\left( 1+   O_{\mu,\nu}(|s|^{-1/2}) \right) 
\end{equation*}
\end{lemma}
\begin{proof}
This is a direct generalization of  Lemma 4.1 in \cite{Matthes1},
which deals with the case $r' = r$.
As in \cite{Matthes1}, we use the integral representation
$$
W_{k,ir}(y) = \frac 1{2\pi i} \int_L 
\frac{\Gamma(\frac12-v-ir)\Gamma(\frac 12-v+ir) \Gamma(v-k)}
{\Gamma(\frac 12 +ir-k)\Gamma(\frac 12-ir-k)} y^v dv
$$
where the the path of integration runs from $-i\infty$ to $i\infty$
and is chosen so that all poles of $\Gamma(v-k)$ are to the left, and
all poles of $\Gamma(\frac 12-v \pm ir)$ are to the right of $L$; this
is possible since we assume that $|k|<1/2$. 
Inserting this into the formula for ${\mathcal M}_{k}(s)$ gives 
\begin{multline*}
{\mathcal M}_{k}(s) = \\ 
\int_0^\infty W_{k,ir'}(y)  \frac 1{2\pi i} \int_L 
\frac{\Gamma(\frac12-v-ir)\Gamma(\frac 12-v+ir) \Gamma(v-k)}
{\Gamma(\frac 12 +ir-k)\Gamma(\frac 12-ir-k)} y^{v+s-2} dv dy
\end{multline*}
Then one uses the formula 
$$ \int_0^\infty W_{k,ir'}(y)e^{-y/2} y^{u-1}dy = 
\frac{\Gamma(\frac 12+u+ir')\Gamma(\frac 12 +u-ir')}{\Gamma(u-k+1)},
\quad \Re(u)>-1
$$
to find  
\begin{eqnarray*}  {\mathcal M}_{k}(s) & = & \frac 1{2\pi i}\int_{L}
\Gamma (-\frac 12 + v + s + ir') \Gamma (-\frac 12 + v + s - ir') \\
& & \times \frac{\Gamma (\frac 12 - v - ir) \Gamma (\frac 12 - v + ir) 
\Gamma (v-k)}{\Gamma( v + s -k) \Gamma (\frac 12 + ir -k) \Gamma
(\frac 12 - ir -k)} dv .
\end{eqnarray*}
One then shifts the contour of integration to the line $\Re(v)=k-1/2$,
picking up a single residue at $v=k$, and estimates the remaining
integral  as in \cite{Matthes1} giving 
$$
\mathcal M_k(s) 
= \frac{\Gamma(s-\frac 12 +k-ir')\Gamma(s-\frac 12 +k+ir')}{\Gamma(s)} 
 + O_{\mu,\nu}\left(\Gamma(s-\frac 32 +2k) \right) 
$$
The conclusion of the Lemma  now follows from Stirling's formula. 
\end{proof}

Let 
\begin{equation}\label{a determinant} 
M(s) = \left[ \begin{array}{cc}  {\mathcal M}_{-k}(s+1) &  {\mathcal M}_{-k}(s)
 \\ {\mathcal M}_{k}(s+1) & -{\mathcal M}_{k}(s)   \end{array} \right]
 \; . 
\end{equation}
\begin{lemma}\label{Nonsingularity of M(s)} 
$M(s)$ is analytic and nonsingular for $\Re(s)>0$. 
\end{lemma}
\begin{proof} 
Holomorphy in $\Re(s)>0$ follows from that of $\mathcal M_k(s)$. 
As in \cite{Matthes1}, one shows that there is a recurrence relation 
\begin{multline}\label{recurrence for M}
(s+1) \mathcal M_k(s+2) -2k(2s+1) \mathcal M_k(s+1)\\ 
=  \frac  1s
\left(s^2-(\mu+\nu)^2 \right)\left(s^2-(\mu-\nu) \right) \mathcal M_k(s)
\end{multline}

 By the recurrence relation \eqref{recurrence for M}, we infer that
\[ (s+1) \det M(s+1) = (s \det M(s)) 
\frac{ (s^{2} - (\mu + \nu)^{2}) \; (s^{2} - (\mu - \nu)^{2})}{s^{2}} \; , \]
and 
\begin{eqnarray*} 
 \det M(s) & = & \frac{\Gamma (s + \mu + \nu)
                       \Gamma (s + \mu - \nu)
                       \Gamma (s - \mu + \nu)
                       \Gamma (s - \mu - \nu)}{
                       \Gamma (s + n + \mu + \nu)
                       \Gamma (s + n + \mu - \nu)
                       \Gamma (s + n - \mu + \nu)
                       \Gamma (s + n - \mu - \nu)} \\
& & \times \frac{(s+n) \Gamma ^{2}(s+n) }{s \; \Gamma ^{2}(s)} 
\det M(s+n) \; . \end{eqnarray*}

By using Lemma~\ref{asymp of Mk} with Stirling's formula, we deduce that\footnote{Note the limit is $-2$ instead of $1$ as in Matthes' paper  
\cite{Matthes1}.} 
\begin{multline*} 
\lim _{n \rightarrow \infty} \frac{(s+n) \Gamma ^{2}(s+n) \det M(s+n) }
{\Gamma (s + n + \mu + \nu)
                       \Gamma (s + n + \mu - \nu)
                       \Gamma (s + n - \mu + \nu)
                       \Gamma (s + n - \mu - \nu)} \\
= -2 \; . 
\end{multline*} 
Therefore we conclude that
\[ \det M(s) = -2 \frac{\Gamma (s + \mu + \nu)
                       \Gamma (s + \mu - \nu)
                       \Gamma (s - \mu + \nu)
                       \Gamma (s - \mu - \nu)}{s \; \Gamma ^{2}(s)}  \; .\]
and thus $\det M(s)\neq 0$ for $\Re(s)>0$. 
\end{proof} 

\subsection{Proof of Proposition~\ref{prop:Matthes}}
We define the Eisenstein series of weight $-2$ and $2$ by
\begin{equation*}
\begin{split}
E_{-2}(z,s) &=  \sum_{\gamma\in \Gamma_\infty \backslash \Gamma_0(4)}
\left(\frac{j(\gamma,z)}{|j(\gamma,z)|}\right)^2 \Im(\gamma z)^s \\
E_{-2}(z,s) &=  \sum_{\gamma\in \Gamma_\infty \backslash \Gamma_0(4)}
\left(\frac{j(\gamma,z)}{|j(\gamma,z)|}\right)^{-2} \Im(\gamma z)^s
\end{split}
\end{equation*}
These series are absolutely convergent for $\Re(s)>1$, with an analytic
continuation (no poles)  to $\Re(s)> 1/2$.

 Using  the Maass raising operator  $\maasK_{1/2}$, we get a form
$\maasK_{1/2}F$ of weight $5/2$ with Fourier expansion
\begin{multline*}
\maasK_{1/2}F(z) = \sum_{n=1}^\infty -\rho(n) W_{5/4,ir}(4\pi ny)e(nx) \\
+\sum_{n=1}^\infty \left((\frac 34)^2+r^2 \right)
\rho(-n)W_{-5/4,ir}(4\pi ny)e(-nx)
\end{multline*}
and using the lowering operator $\maasL_{1/2}$ we get a form
$\maasL_{1/2}F$ of weight  $-3/2$ with Fourier expansion
\begin{multline*}
\maasL_{1/2}F(z) = \sum_{n=1}^\infty -\left( (\frac 14)^2+r^2 \right)
\rho(n) W_{-3/4,ir}(4\pi ny)e(nx) \\
+\sum_{n=1}^\infty \rho(-n)W_{3/4,ir}(4\pi ny)e(-nx)
\end{multline*}
 
Consider the Rankin-Selberg integrals
$$
\mathcal A(s) =\int_{\Gamma_0(4) \backslash \HH}
F(z)\overline{F'(z)}  E(z,s) \frac{dxdy}{y^2}
$$
and
\begin{multline*}
2\mathcal B(s)=
\int_{\Gamma_0(4) \backslash \HH} \maasK_{1/2} F(z)\overline{F'(z)} E_{-2}(z,s)
\frac{dxdy}{y^2} \\
+ \int_{\Gamma_0(4) \backslash \HH} \maasL_{1/2} F(z)\overline{F'(z)}
E_{2}(z,s) \frac{dxdy}{y^2}
\end{multline*}
Since $E_{\pm 2}(z,s)$ are  analytic in  $\Re(s)>1/2$,
$\mathcal B(s)$ is analytic in $\Re(s)>1/2$; and since $E(z,s)$ is
analytic in $\Re(s)>1/2$ except
for  simple pole at $s=1$, $\mathcal A(s)$ is  analytic for
$\Re(s)>1/2$ except for  possibly a simple pole at $s=1$, where the
residue is
$$
\mbox{Res}_{s=1} \mathcal A(s) =
\frac{\langle F,F' \rangle}{\Vol(\Gamma_0(4)\backslash \HH)} =
\frac{\langle F,F' \rangle}{2\pi}
$$
so that $\mathcal A(s)$ is analytic also at $s=1$ if $F$ and $F'$ are
orthogonal.

By the standard unfolding trick, we find that for $\Re(s)>1$,
\begin{multline*}
\mathcal A(s) = \sum_{n=1}^\infty
\frac{4\pi n\rho(n)\overline{\rho'(n)}}{(4\pi n)^{s}}
\int_0^\infty W_{1/4,ir}(y)W_{1/4,ir'}(y) y^{s-2}dy \\
+ \sum_{n=1}^\infty \frac{4\pi n\rho(-n)\overline{\rho'(-n)}}{(4\pi n)^{s}}
\int_0^\infty W_{-1/4,ir}(y)W_{-1/4,ir'}(y) y^{s-2}dy
\end{multline*}
and
\begin{multline*}
2\mathcal B(s) = -\sum_{n=1}^\infty
\frac{4\pi n\rho(n)\overline{\rho'(n)}}{(4\pi n)^{s}}
\int_0^\infty W_{5/4,ir}(y)W_{1/4,ir'}(y) y^{s-2}dy \\
+ \left( (\frac 34)^2+r^2 \right)
\sum_{n=1}^\infty \frac{4\pi n \rho(-n)\overline{\rho'(-n)}}{(4\pi n)^{s}}
\int_0^\infty W_{-5/4,ir}(y) W_{-1/4,ir'}(y) y^{s-2}dy
\end{multline*}

Setting $k = 1/4$ we then have
$$ \mathcal A(s) = (4\pi)^{-s} \left(L^{+}(s)  {\mathcal M}_{k}(s) 
+ L^{-}(s)  {\mathcal M}_{-k}(s)\right) 
$$
and moreover
\begin{lemma}
\begin{multline}
 \mathcal B (s) = (4\pi)^{-s} (L_{+}(s) ( r/2 { \mathcal M}_{k}(s) - 1/2 
{\mathcal M}_{k}(s+1))\\
+ L_{-}(s)(r/2 {\mathcal M}_{-k}(s) \; + \; 1/2 {\mathcal M}_{-k}(s+1)))  
\end{multline}
\end{lemma}
\begin{proof}
Let $k > 0$ be any half integer. In our case $k = 1/2$.
We have, by unfolding the integral, that
\[ \int _{\Gamma _{0}\backslash \HH} K_{k} F(z) \overline{F'(z)}  E_{-2}(z,\;s)
\frac{dx dy}{y^{2}}
= \int ^{\infty}_{0} \int ^{1}_{0} K_{k} F(z) \overline{F'(z)} y^{s-2} dx dy
,\]
and
\[ \int _{\Gamma _{0}\backslash \HH} \Lambda _{k} F(z) \overline{F'(z)}  
E_{2}(z,\;s)
\frac{dx dy}{y^{2}}
= \int ^{\infty}_{0} \int ^{1}_{0} \Lambda _{k} F(z) \overline{F'(z)} 
y^{s-2} dx dy .\]
Now
\begin{eqnarray*}
K_{k} W_{\frac{k}{2} sgn(n),\; ir}(4\pi |n|y) e(nx) &  = &
\left[ ( k/2 - 2 \mbox{sign}(n) \pi |n| y)  
W_{\frac{k}{2} sgn(n),\; ir}(4\pi |n|y) \right.
\\ & &  \left. \mbox{} +  4\pi |n|y 
W^{'}_{\frac{k}{2} sgn(n),\; ir}(4\pi |n|y) \right] e(nx) , \end{eqnarray*}
and
\begin{eqnarray*}
\Lambda_{k} W_{\frac{k}{2} sgn(n),\; ir}(4\pi |n|y) e(nx) &  = &
\left[ ( k/2 - 2 \mbox{sign}(n) \pi |n| y)  
W_{\frac{k}{2} sgn(n),\; ir}(4\pi |n|y) \right.
\\ & &  \left. \mbox{} -  4\pi |n|y 
W^{'}_{\frac{k}{2} sgn(n),\; ir}(4\pi |n|y) \right] e(nx) . \end{eqnarray*}
Hence Lemma 5.5 follows. 
\end{proof}
 
Consequently, 
\begin{equation*}
\begin{split}
 \mathcal A(s) &= (4\pi)^{-s} \left( L_{+}(s)  {\mathcal M}_{k}(s) 
+ L_{-}(s)  {\mathcal M}_{-k}(s) \right) \\
 r \mathcal A(s) - 2 \mathcal B(s) &= 
(4\pi)^{-s} \left( L_{+}(s)  {\mathcal M}_{k}(s+1) 
- L_{-}(s)  {\mathcal M}_{-k}(s+1) \right) 
\end{split}
\end{equation*}
Solving for $L_{+}(s)$, $L_{-}(s)$, we obtain
\[ \left[ \begin{array}{c} L_{+}(s) \\ L_{-}(s) \end{array} \right]   
= - \frac{(4\pi)^{s}}{\det M(s)} M(s) 
\left[ \begin{array}{l} \mathcal A(s) \\  r \mathcal A(s) - 2 \mathcal B(s)  \end{array} \right] \; ,
\]
where $M(s)$ is given by \eqref{a determinant}. 
Therefore Proposition~\ref{prop:Matthes}  follows 
from Lemma~\ref{Nonsingularity of M(s)}.

\subsection{Proof of Proposition~\ref{prop:Rankin-Selberg}}

Consider the Rankin-Selberg integrals
$$
\mathcal I(s) =\int_{\Gamma_0(4) \backslash \HH}
F(z)\overline{F(z)}  E(z,s) \frac{dxdy}{y^2}
$$
and
$$
\mathcal J(s)=
\int_{\Gamma_0(4) \backslash \HH} \maasK_{1/2} F(z)\overline{F(z)} E_{-2}(z,s)
\frac{dxdy}{y^2}
$$
Since $E_{-2}(z,s)$ is  analytic in  $\Re(s)>1/2$,
$\mathcal J(s)$ is analytic in $\Re(s)>1/2$; and since $E(z,s)$ except
for  simple pole at $s=1$, $\mathcal I(s)$ is  analytic for
$\Re(s)>1/2$ except for a simple pole at $s=1$, where the
residue is
$$
\mbox{Res}_{s=1} \mathcal I(s) =
\frac{\langle F,F \rangle}{\Vol(\Gamma_0(4)\backslash \HH)} =
\frac{\langle F,F \rangle}{2\pi}
$$

By the standard unfolding trick, we find that for $\Re(s)>1$,
\begin{multline}\label{unfold I}
\mathcal I(s) = \sum_{n=1}^\infty
\frac{|\rho(n)|^2}{(4\pi n)^{s-1}}
\int_0^\infty W_{1/4,ir}(y)^2 y^{s-2}dy \\
+ \sum_{n=1}^\infty \frac{|\rho(-n)|^2}{(4\pi n)^{s-1}}
\int_0^\infty W_{-1/4,ir}(y)^2 y^{s-2}dy
\end{multline}
and
\begin{multline}\label{unfold J}
\mathcal J(s) = -\sum_{n=1}^\infty
\frac{|\rho(n)|^2}{(4\pi n)^{s-1}}
\int_0^\infty W_{5/4,ir}(y)W_{1/4,ir}(y) y^{s-2}dy \\
+ \left( (\frac 34)^2+r^2 \right)
\sum_{n=1}^\infty \frac{|\rho(-n)|^2}{(4\pi n)^{s-1}}
\int_0^\infty W_{-5/4,ir}(y) W_{-1/4,ir}(y) y^{s-2}dy
\end{multline}

Denoting by $R^{\pm}$ the residue at $s=1$ of $L_{\pm}(s)$, we find
from \eqref{unfold J} that since $\mathcal J(s)$ is analytic at $s=1$
\begin{multline*}
 \left( (\frac 34)^2+r^2 \right)\int_0^\infty W_{-5/4,ir}(y)
W_{-1/4,ir}(y)\frac{dy}{y}  R^{-}\\ =
\int_0^\infty W_{5/4,ir}(y)W_{1/4,ir}(y)\frac{dy}{y}  R^{+}
\end{multline*}

 From the formula \cite[p. 858]{GR} that for $|\Re (\mu)| < 1/2$,
\begin{multline*}
\int ^{\infty}_{0} W_{\kappa,\;\mu}(x)
W_{\lambda,\;\mu}(x) \frac{dx}{x}
 =   \frac{1}{(\kappa - \lambda) \sin 2\pi \mu} \\
\times \left[ \frac{1}{\Gamma (1/2  - \kappa + \mu) \;
\Gamma (1/2 - \lambda - \mu)} -
\frac{1}{\Gamma (1/2  - \kappa - \mu) \;
\Gamma (1/2 - \lambda + \mu)} \right] \; ,
\end{multline*}
it follows that
\[\int ^{\infty}_{0}
W_{\frac{5}{4},\;ir}(x)
W_{\frac{1}{4},\;ir}(x)\frac{dx}{x} = \frac{1}{\sinh (2r\pi)}
\frac{2r}{|\Gamma (1/4 + ir)|^{2}} \; ,\]
and
\[\int ^{\infty}_{0}
W_{-\frac{5}{4},\;ir}(x)
W_{-\frac{1}{4},\;ir}(x) \frac{dx}{x}= \frac{1}{\sinh (2r\pi)}
\frac{2r}{|\Gamma (7/4 + ir)|^{2}} \; .\]
Hence
\begin{equation}\label{rel R}
R^{+} =
\frac{|\Gamma (1/4 + ir)|^{2}}{|\Gamma (3/4 + ir)|^{2}}
R^{-} \; .
\end{equation}

Computing the residue at $s=1$ of $\mathcal I(s)$ using \eqref{unfold I} gives
\begin{equation}\label{res I}
\frac{4\pi \langle F,F \rangle }{\Vol(\Gamma_0(4)\backslash \HH)}  =
R^+\int_0^\infty W_{\frac 14,ir}(y)^2 \frac{dy}{y} +
R^-\int_0^\infty W_{-\frac 14,ir}(y)^2 \frac{dy}{y}
\end{equation}

We have for $|\Re(\mu)| < 1/2$ \cite[p. 858]{GR} :
$$
\int ^{\infty}_{0} W_{\kappa,\;\mu}^{2}(z) \frac{dz}{z}
 = \frac{\pi}{\sin 2\pi \mu} \frac{\psi (1/2 + \mu - \kappa) -
\psi (1/2 - \mu - \kappa)}{\Gamma (1/2 + \mu - \kappa) \;
\Gamma (1/2 - \mu - \kappa)}
$$
where $\psi (x) = \frac{d}{dx} \log \Gamma (x)$.
Therefore
\begin{equation}\label{intW+}
\int_0^\infty W_{\frac 14,ir}(y)^2 \frac{dy}{y} =
\frac{\pi}{i\sinh 2\pi r} \frac{\psi (1/4 + ir ) -
\psi (1/4 - ir)}{|\Gamma (1/4 + ir)|^2}
\end{equation}
and
\begin{equation}\label{intW-}
\int_0^\infty W_{-\frac 14,ir}(y)^2 \frac{dy}{y} =
\frac{\pi}{i\sinh 2\pi r} \frac{\psi (3/4 + ir ) -
\psi (3/4 - ir)}{|\Gamma (3/4 + ir)|^2}
\end{equation}

Inserting \eqref{intW+}, \eqref{intW-} and \eqref{rel R} into
\eqref{res I} gives
\begin{multline}
2\langle F,F \rangle =
R^+ \frac{\pi}{i\sinh(2\pi r) |\Gamma(\frac 14+ir)|^2} \\
\times \left(
\psi (\frac 14 + ir ) -\psi (\frac 14 - ir)+\psi (\frac 34 + ir ) -
\psi (\frac 34 - ir) \right)
\end{multline}
Taking the logarithmic derivative of the reflection formula
$$
\Gamma(x)\Gamma(1-x)=\frac \pi{\sin(\pi x)}
$$ 
gives
$$ \psi(x)-\psi(1-x) = -\pi\cot(\pi x)$$
Hence we find
\begin{multline*}
\psi (\frac 14 + ir) - \psi (\frac 14 - ir) + \psi (\frac 34 + ir)
- \psi (\frac 34 - ir) \\
 =  -\frac{\pi \cos \pi (\frac 14 + ir)}{\sin \pi (\frac 14 + ir)}
+ \frac{\pi \cos \pi (\frac 14 - ir)}{\sin \pi (\frac 14 - ir)} \\
 =  \frac{\pi i \sinh (2\pi r)}{|\sin \pi (\frac 14 + ir)|^{2}}
 =   \frac{2\pi i \sinh (2\pi r)}{\cosh (2\pi r)} \\
 =   2 i \sinh (2\pi r) |\Gamma (\frac 12 + 2ir)|^{2} \;.
\end{multline*}
Therefore we get
\begin{equation*}
\langle F,F \rangle   = \pi \frac{|\Gamma(\frac
12+2ir)|^2}{|\Gamma(\frac 14+ir)|^2} \cdot R^+
\end{equation*}
which proves \eqref{form2 for R+}.

\section{Theta lifts and periods}\label{sec theta}

We summarize the results on theta lifts of Shintani \cite{Shintani} and
Kohnen \cite{Ko1,Ko2} in the holomorphic case, and of Katok and Sarnak
\cite{KS} in the Maass case.

\subsection{Holomorphic forms}
For a holomorphic form $f$ of weight $2k$ (we will later take only
even $k$) for the full modular group $\SL_2(\Z)$, Shintani
\cite{Shintani} showed that
it can be lifted to a cuspidal Hecke eigenform $\theta(f,z)\in S_{k+\frac
12}(\Gamma_0(4),\chi_k)$  of weight $k+\frac 12$
with character $\chi_k(\begin{pmatrix}a&b\\c&d\end{pmatrix}) = (\frac
{-1}d)^k$ for $\Gamma_0(4)$, that is transforming as
$$ F(\gamma z) = \chi_k(\gamma)J_1(\gamma,z)^{2k+1} F(z) ,\quad
\gamma\in \Gamma_0(4) 
$$
where $J_1(\gamma,z) = \theta_1(\gamma z)/\theta_1(z)$, $\theta_1(z) =
\sum_{n=-\infty}^\infty e(n^2z)$.
Thus
$$J_1(\begin{pmatrix}a&b\\c&d\end{pmatrix},z) = \epsilon_d^{-1} (\frac
cd) \sqrt{cz+d}
$$
with $\epsilon_d=1$ if $d=1\mod 4$ and $\epsilon_d=i$ if $d=3\mod 4$.
In particular $J_1(\gamma,z)^2 = (\frac {-1}d)(cz+d)$.

We write the Fourier expansion of $\theta(f,z)$ is
$$\theta(f,z) = \sum_{d\geq 1}c_f(d) e(dz)$$
Then
$$c_f(d)  = \sum_{\disc(q)=d} \int_{C(q)} f(w)(a-bw+cw^2)^{k-1}dw $$
the sum over all $\SL_2(\Z)$-equivalence classes of  binary quadratic forms
$q=[a,b,c]$ of positive discriminant $d=b^2-4ac$.

We note that from \eqref{exp for period via forms} 
$$  \int_{C(q)} f(w)(a-bw+cw^2)^{k-1}dw = d^{\frac{k-1}2} \int_{C(q)}
f d\mu_{C(q)} $$
is a simple multiple of the period of $f$ over the closed (primitive,
oriented) geodesic $C(q)$ corresponding to $q$. 
Thus
$$
\frac{c_f(d)}{d^{(k-\frac 12)/2}}
= \frac 1{d^{1/4}} \sum_{\disc q=d}\int_{C(q)}fd\mu_{C(q)}   
$$
(the sum over all forms of discriminant $d$, not necessarily
primitive), that is
\begin{equation}\label{mu as fourier coeff} 
 \frac{c_f(d)}{d^{(k-\frac 12)/2}} =\frac {\mu_d(f)}{d^{1/4}}
\end{equation}

\subsubsection{An inner product formula} 
Denote by $S^{+}_{k+1/2}(\Gamma _{0}(4))$ the space of cusp forms of
weight $k+1/2$ on $\Gamma _{0}(4)$, whose Fourier expansion  
$\sum _{n\geq 1} c(n)e(nz)$ satisfies the condition $c(n) = 0$ unless 
$(-1)^{k}n \equiv 0,1 \mod 4 $. It was proved by Kohnen
\cite{Ko1} that the two spaces $S_{2k}(\Gamma _{0}(1))$ and
$S^{+}_{k+1/2}(\Gamma _{0}(4))$  are
isomorphic as modules over the Hecke algebra under the Shimura
correspondence. Assume $k$ is even, let 
$$f(z)=\sum_{n=1}^\infty a_f(n) e(nz)\in S_{2k}(\Gamma _{0}(1))$$ 
and normalize the L-function of $f$ as 
$$L(s,f) = \sum_{n=1}^\infty \frac{a_f(n)}{n^{s+(2k-1)/2}} 
$$ 
so that the
functional equation is $s\mapsto 1-s$. 
Let 
$$h_{f}(z) = \sum _{n\geq 1} c_{h_{f}}(n) e(nz) \in S^{+}_{k+1/2}(\Gamma _{0}(4))$$  
corresponds to $f$ as above such that
$\langle h_{f}, h_{f}\rangle = 1$. 
By the works of Kohnen \cite[Theorem 3 and Corollary 1]{Ko2}, we have
\[ c_{h_{f}}(m) \overline{c_{h_{f}}(1)} = 
\frac{(-1)^{k/2}2^{k}}{\langle f, f\rangle} a_{f}(m) \;, \]
and
\[ |c_{h_{f}}(1)|^{2} = \frac{(k-1)!}{\pi ^{k}} 
\frac{L(\frac 12, f)}{\langle f, f\rangle} \; .\]
Hence we see that $\theta (f, z) \neq 0$ if and only if 
$L(\frac 12, f) \neq 0$,
and in this case $h_{f}$ is proportional to $\theta (f, z)$: 
$$
\overline{c_{h_{f}}(1)} h_{f}(z) = \frac{(-1)^{k/2}2^{k}}{\langle f, f\rangle}
\theta (f, z)
$$
Thus we get a formula for the inner product of the lift: 
If $g$ is another cuspidal Hecke eigenform in
$S_{2k}(\Gamma _{0}(1))$, then
\begin{equation} \label{Kohnen inner product} 
\langle \theta (f, \cdot), \; \theta (g, \cdot) \rangle
 = 
\frac{(k-1)!}{2 ^{2k} \pi ^{k}} L(\frac 12, f) \langle f, g\rangle  
\end{equation}

\subsection{Maass forms}
Given an even Maass form $\phi$ for $\SL_2(\Z)$, with eigenvalue
$\lambda = \frac 14+(2r)^2$, the theta lift $F(z)=\theta(\phi,z)$ is a Maass form for
$\Gamma_0(4)$ of weight $1/2$, that is transforming under $\Gamma_0(4)$
as
$$ F(\gamma z) = J(\gamma,z) F(z),\quad \gamma\in \Gamma_0(4)$$
where $J(\gamma,z) = \theta(\gamma z)/\theta(z)$, $\theta(z)
=y^{1/4}\theta_1(z)= y^{1/4}\sum_{n=-\infty}^\infty e(n^2z)$.
Moreover $F$ is an eigenfunction of $\Delta_{1/2}$ with eigenvalue $\frac
14 +r^2$.

The Fourier expansion of $F$ is given by \cite{KS} as 
$$
F(u+iv) = \sum_{d\neq 0} \rho(d)
W_{\frac{\sign(d)}4,ir}(4\pi|d|v) e(dv)
$$
where for $d>0$ 
\begin{equation*}
\rho(d) =
\frac{1}{\sqrt{8}\pi ^{1/4} d ^{3/4}}
\sum_{\disc(q)=d} \int _{C(q)} \phi  ds 
\end{equation*} 
the sum over all $\SL_2(\Z)$-equivalence classes of  binary quadratic forms
$q=[a,b,c]$ of positive discriminant $d=b^2-4ac$ and $C(q)$ is 
the closed (primitive, oriented) geodesic  corresponding to $q$.
(For $d<0$ there is an analogous formula involving Heegner points). 
Thus 
\begin{equation} \label{KS coeff formula}
\rho(d) = \frac 1{\sqrt{8}\pi^{1/4} d^{3/4}} \mu_d  
\end{equation}


Moreover if $\psi$ is another Hecke-Maass eigenform,
we have, in view of  \cite[formula (5.6), p. 224]{KS}, the inner
product formula 
\begin{equation}\label{inner product Maass} 
\langle\theta (\phi, \cdot),\;\theta (\psi, \cdot)\rangle =
\frac{3}{2} \Lambda (1/2,\;\phi) \langle \phi ,\;\psi\rangle ,
\end{equation}
where
\[\Lambda (s,\;\phi) = \pi ^{-s}
\Gamma \left( \frac{s + 2ir}{2} \right)
\Gamma \left( \frac{s - 2ir}{2} \right) L(s,\; \phi) .\]
Note the above formula is still valid even if $\phi, \;\psi$ are not
even.

\section{Proof of the main Theorem}\label{sec:proof of main thm} 

\subsection{Expected value of $\mu_d$}\label{sec:expected value}
In section~\ref{sec theta} we identified the measures $\mu_d(f)$ 
with Fourier coefficients of theta-lifts up to a normalization. Hence
the vanishing of the mean value of $\mu_d(f)/d^{1/4}$ follows from the
corresponding fact for Fourier coefficients of forms of half-integer
weight. 

The first statement of Theorem~\ref{thm 1} follows immediately from the
Hecke's bound that for any
$\alpha \in \R$ and $\epsilon > 0$,  we have
\[  \sum _{n \leq  N} a(n) e(\alpha n) = O(N^{1/2 + \epsilon})  \; , \]
where $a(n)$ is the normalized Fourier coefficient of any holomorphic or Maass
cusp
form, and the implicit constant depends on the form and $\epsilon$ alone.
For the proof, see for example 
\cite[page 111, Theorem 8.1]{Iw1} and
\cite[page 71, Theorem 5.3]{Iw2}.
One makes use of the following formula 
(\cite[page 857, 7.611]{GR}: \\
\[ \int ^{\infty}_{0} x^{-1} W_{k,\;\mu}(x) dx \; = \;
\frac{\pi ^{3/2} 2^{k} \sec (\mu \pi)}{\Gamma (\frac 34 - \frac k2 + \frac \mu 2)
\Gamma (\frac 34 - \frac k2 - \frac \mu 2)}  \; .\]

\subsection{Proof of Theorem~\ref{main thm}}

We wish to compute the bilinear form 
$$B(f,g)=\lim_{N\to\infty}  
\frac 1{\#\{d:d\leq N\}} \sum_{d\leq N} \frac{\mu_d(f)\overline{\mu_d(g)}} {d^{1/2}} \;,
$$
where $f,g\in L^2_{cusp}(\Gamma\backslash G)$ are smooth and $K$-finite, and the sum is over discriminants.
We take $f$ and $g$ to lie in the irreducible subspaces $U_\pi$ defined in \eqref{def of U}, 
that is subspaces of $L^2_{cusp}(\Gamma\backslash G)$ irreducible under the joint actions of $G$, 
the orientation reversal symmetry $\ors$ and under the Hecke algebra. We wish to show that for 
such $f$, $g$, 
\begin{equation}\label{diffs orthogonal} 
 B(f,g)=0
\end{equation}
if $f$, $g$ lie in distinct (hence orthogonal) subspaces $U_f$, $U_g$, and to compute $B(f,f)$. 
By the results of \S~\ref{ladder}, it suffices to consider ``minimal'', or generating vectors, 
that is to consider holomorphic forms or Maass forms. 
In particular we need to show that for such $f$, 
\begin{equation}\label{final result} 
 \lim_{N\to\infty}  
\frac 1{\#\{d:d\leq N\}}\sum_{d\leq N} \frac{|\mu_d(f)|^2}{d^{1/2}} = c(f) L(\frac 12,f) V^{sym}(f,f)
\end{equation}
where 
$$c(f) = 
\begin{cases}
   6/\pi,& f \mbox{ Maass form  }  \\ 1/ \pi & f \mbox{ holomorphic   }  
\end{cases}
$$

Since both $B$ and $V^{sym}$ vanish when $f$ is holomorphic of weight $\equiv 2\mod 4$ 
or an odd Maass form, it suffices to treat the cases of holomorphic forms of weight divisible by $4$ 
and of even Maass forms. 
To do so, we recall that in these cases we have identified $\mu_d(f)$ with
simple  multiples of the Fourier coefficients of theta-lifts
$\theta(f,z)$.   
Thus we may use  Rankin-Selberg theory (Corollary~\ref{Tauberian cor}) 
to recover \eqref{diffs orthogonal}  and \eqref{final result} once we
have made the correct identifications. 
We treat separately the case of holomorphic forms and Maass forms.

\subsubsection{Holomorphic forms}
For a cuspidal Hecke eigenform of weight $2k$, $k$ even, 
the theta lift has weight $k+1/2$ with 
Fourier expansion 
$\theta(f,z) = \sum_{d\geq 1}c_f(d) e(dz)$
with \eqref{mu as fourier coeff} 
$$
 \frac{c_f(d)}{d^{(k-\frac 12)/2}} =\frac {\mu_d(f)}{d^{1/4}}
$$

By standard Rankin-Selberg theory \eqref{standard RS} 
$$ (4\pi)^{-(s+k-\frac 12)} \Gamma(s+k-\frac 12) 
\sum_{n=1}^\infty|\frac{c(n)}{n^{\frac{k-1/2}2}}|^2 n^{-s} 
$$ 
has a simple pole at $s=1$ with residue 
$$R^+(f)=\frac{\langle \theta(f),\theta(f)\rangle}{2\pi}= 
\frac 1{2\pi} (4\pi)^{-k}\Gamma(k) L(\frac 12,f) \langle f,f \rangle
$$ 
by \eqref{Kohnen inner product} and hence the Dirichlet series 
$$ D(s)=\sum_{n=1}^\infty |\frac{\mu_n(f)}{n^{1/4}}|^2 n^{-s}$$
has a simple pole at $s=1$ with residue 
$$ \frac{(4\pi)^{k+\frac 12} }{\Gamma(k+\frac 12)} R^+(f) = 
\frac {\Gamma(k)}{\sqrt{\pi}\Gamma(k+\frac 12)} L(\frac 12,f) \langle
  f,f \rangle \;.
$$
Note that by \eqref{determine V^sym discrete} 
$$ 
V^{sym}(f,f) = \frac 12 2^{2k}\frac{\Gamma(k)^2}{\Gamma(2k)} \langle f,f\rangle  = 
\pi\frac{\Gamma(k)}{\sqrt{\pi}\Gamma(k+\frac 12)} 
\langle f,f\rangle 
$$
and therefore the residue at $s=1$ of $D(s)$ is  
$$
 \frac 1\pi V^{sym}(f,f)L(\frac 12, f) \;.
$$
Consequently we find 
$$ \lim_{N\to \infty} \frac 1N\sum_{n\leq N}
|\frac{\mu_n(f)}{n^{1/4}}|^2 = \frac 1{\pi}  V^{sym}(f,f)L(\frac 12, f) 
$$

\subsubsection{Maass forms} 
 Let $\phi$ be an even Maass Hecke
eigenform with Laplace eigenvalue $1/4+(2r)^2$ and
$F=\theta(\phi,\cdot)$ its theta-lift, with Fourier expansion
$$
F(u+iv) = \sum_{d\neq 0} \rho(d)W_{\frac{\sign(d)}4,ir}(4\pi|d|v) e(dv)
$$ 
Recall \eqref{KS coeff formula} that for $d>0$ we have
$$
\rho(d) =
\frac{\sqrt{2}}{4\pi ^{1/4} d ^{3/4}} \mu_d
$$
Thus by Corollary~\ref{Tauberian cor} we have
$$\frac 1N \sum_{1\leq n\leq N} \left( \frac{ \mu_n(\phi)}{n^{1/4}} \right)^2
\sim  \frac 2{\sqrt{\pi} } R^+(\phi)
$$
where $R^+(\phi)$ is given by \eqref{form2 for R+}.

Inserting \eqref{inner product Maass} into \eqref{form2 for R+} gives
\begin{equation}
R^+ = \frac{3}{2\pi^{3/2}} \frac{|\Gamma(\frac
14+ir)|^4}{|\Gamma(\frac 12+2ir)|^2} L(\frac 12,\phi) \langle
\phi,\phi \rangle
\end{equation}
We note that by \eqref{determine V^sym spherical}  
\[ \frac{ | \Gamma (\frac 14 + ir)|^{4}}{ 2\pi |\Gamma (\frac 12 + 2ir)|^{2}}
\langle\phi,\phi\rangle \; =  \; V^{sym}(\phi,\phi) \; , \]
and hence
\begin{equation}
R^+ = \frac {3}{\sqrt{\pi} } V^{sym}(\phi,\phi) L(\frac 12,\phi)
\end{equation} 
Therefore we get
\begin{equation}
\frac 1N \sum_{1\leq n\leq N} \left( \frac{ \mu_n(\phi)}{n^{1/4}} \right)^2
\sim  \frac 2{\sqrt{\pi} } R^+ = \frac {6}{\pi} V^{sym}(\phi,\phi)
L(\frac 12,\phi)
\end{equation}

 \subsubsection{Orthogonality} 
Finally, the fact that $B(f,g)=0$ if the subspaces $U_f$, $U_g$ are distinct follows 
from standard Rankin-Selberg theory when at least one of $f$ or $g$ is holomorphic, see  
\eqref{standard RS orthogonal} and \eqref{standard RS orthogonal maass holom},   
while the case of both $f$, $g$ being Maass forms follows from Corollary~\ref{Tauberian cor}.

\end{document}